\documentclass[10pt, a4paper, twoside]{amsart}

\makeatletter
\def\@settitle{\begin{center}%
		\baselineskip14\p@\relax
		\normalfont\LARGE\scshape\bfseries
		\@title
	\end{center}%
}
\makeatother

\makeatletter

\def\subsection{\@startsection{subsection}{2}%
	\z@{.5\linespacing\@plus.7\linespacing}{.5\linespacing}%
	{\normalfont\large\bfseries}}

\def\subsubsection{\@startsection{subsubsection}{3}%
	\z@{.5\linespacing\@plus.7\linespacing}{.5\linespacing}%
	{\normalfont\itshape}}

\usepackage[usenames, dvipsnames]{color}
\definecolor{darkblue}{rgb}{0.0, 0.0, 0.45}

\usepackage[colorlinks	= true,
raiselinks	= true,
linkcolor	= darkblue, 
citecolor	= Mahogany,
urlcolor	= ForestGreen,
pdfauthor	= {Peyman Mohajerin Esfahani},
pdftitle	= {},
pdfkeywords	= {},
pdfsubject	= {},
plainpages	= false]{hyperref}

\usepackage[table]{xcolor}	
\usepackage{dsfont,amssymb,amsmath,subfigure, graphicx,enumitem,multirow} 
\usepackage{amsfonts,dsfont,mathtools, mathrsfs,amsthm} 
\usepackage[amssymb, thickqspace]{SIunits}
\usepackage{algorithm,algorithmicx,fancyhdr}
\usepackage{subfigure,epstopdf,todonotes}

\allowdisplaybreaks
\date{\today}

\theoremstyle{theorem}
\newtheorem{Thm}{Theorem}[section]
\newtheorem{Prop}[Thm]{Proposition}

\newtheorem{Cor}[Thm]{Corollary}
\newtheorem{As}[Thm]{Assumption}

\newtheorem{Def}[Thm]{Definition}

\newtheorem{Rem}[Thm]{Remark}

\theoremstyle{remark}
\newtheorem{Ex}{Example}

\newcommand{\Max}{\max\limits_}

\newcommand{\Sup}{\sup\limits_}
\newcommand{\Inf}{\inf\limits_}

\newcommand{\R}{\mathbb{R}}
\newcommand{\PP}{\mathds{P}}
\newcommand{\EE}{\mathds{E}}

\newcommand{\ra}{\rightarrow}

\newcommand{\Let}{\coloneqq}

\newcommand{\diff}{\mathrm{d}}

\newcommand{\wt}{\widetilde}
\newcommand{\wh}{\widehat}

\newcommand{\tr}{^\intercal}
\newcommand{\opt}{^\star}

\newcommand{\eps}{\varepsilon}

\newcommand{\X}{\mathbb{X}}

\newcommand{\F}{\mathcal{F}}

\newcommand{\amb}{\wh{\mathcal{P}}}
\newcommand{\xid}{\wh{\xi}}
\newcommand{\xd}{\wh{x}}
\newcommand{\sd}{\wh{s}}
\newcommand{\Pem}{{\wh{\PP}_N}}

\newcommand{\Q}{\mathds{Q}}
\newcommand{\M}{\mathcal{M}}

\newcommand{\dir}[1]{\delta_{#1}}

\newcommand{\Wass}[3]{{\mathrm W}_{#1}\big(#2,#3\big)}
\newcommand{\inner}[2]{\big \langle #1, #2 \big \rangle }

\newcommand{\ball}[3]{\mathbb{B}^{#1}_{#3}(#2)} 

\newcommand{\st}{\normalfont \text{s.t.}}

\newcommand{\cvar}{{\rm CVaR}}
\newcommand{\var}{{\rm VaR}}

\newcommand{\cert}{J}

\newcommand{\risk}{\rho}

\newcommand{\gec}[1]{\succeq_{#1}}

\newcommand{\coneX}{\mathcal{K}}
\newcommand{\coneXi}{\mathcal{C}}

\newcommand{\gesdp}{\succeq}

\newcommand{\s}{\mathbb{S}}
\newcommand{\Mini}{\mathop{\text{minimize}}}

\DeclareMathOperator*{\argmin}{arg\,min}

\newcommand{\ch}[1]{{\rm \bf C}#1}

\let\P\PP
\let\S\s

\title[Data-driven Inverse Optimization with {Imperfect} Information]
{Data-driven Inverse Optimization with {Imperfect} Information}

\author[P. Mohajerin Esfahani, S. Shafieezadeh-Abadeh, G.A. Hanasusanto, D. Kuhn]{Peyman Mohajerin Esfahani, Soroosh Shafieezadeh-Abadeh, \\ Grani A. Hanasusanto, and Daniel Kuhn}%
	\thanks{The authors are with the Delft Center for Systems and Control, TU Delft, The Netherlands ({\tt P.MohajerinEsfahani@tudelft.nl}), the Risk Analytics and Optimization Chair, EPFL, Switzerland (\texttt{\{soroosh.shafiee,daniel.kuhn\}@epfl.ch}) and the Graduate Program in Operations Research and Industrial Engineering, UT Austin, USA (\texttt{grani.hanasusanto@utexas.edu}).}

\begin{document}
\maketitle

\begin{abstract}
	In data-driven inverse optimization an observer aims to learn the preferences of an agent who solves a parametric optimization problem depending on an exogenous signal. Thus, the observer seeks the agent's objective function that best explains a historical sequence of signals and corresponding optimal actions. {We focus here on situations where the observer has imperfect information, that is, where the agent's true objective function is not contained in the search space of candidate objectives, where the agent suffers from bounded rationality or implementation errors, or where the observed signal-response pairs are corrupted by measurement noise. We formalize this inverse optimization problem as a distributionally robust program minimizing the worst-case risk that the {\em predicted} decision ({\em i.e.}, the decision implied by a particular candidate objective) differs from the agent's {\em actual} response to a random signal. We show that our framework offers rigorous out-of-sample guarantees for different loss functions used to measure prediction errors and that the emerging inverse optimization problems can be exactly reformulated as (or safely approximated by) tractable convex programs when a new suboptimality loss function is used. We show through extensive numerical tests that the proposed distributionally robust approach to inverse optimization attains often better out-of-sample performance than the state-of-the-art approaches.}
\end{abstract}

\section{Introduction} 

In inverse optimization an observer aims to learn the preferences of an agent who solves a parametric optimization problem depending on an exogenous signal. The observer knows the constraints imposed on the agent's actions but is unaware of her objective function. By monitoring a sequence of signals and corresponding actions, the observer seeks to identify an objective function that makes the observed actions optimal in the agent's optimization problem. This learning problem can be cast as an {\em inverse optimization problem} over candidate objective functions. The hope is that the solution of this inverse problem enables the observer to predict the agent's future actions in response to new signals.

Inverse optimization has a wide spectrum of applications spanning several disciplines ranging from econometrics and operations research to engineering and biology. For example, a marketing executive aims to understand the purchasing behavior of consumers with unknown utility functions by monitoring sales figures~\cite{ackerberg2007econometric,bertsimas2014data,bajari2004estimating}, a transportation planner wishes to learn the route choice preferences of the passengers in a multimodal transport system by measuring traffic flows~\cite{ahuja2000faster,ref:Bra-05,burton1992instance,farago2003inverse}, or a healthcare manager seeks to design clinically acceptable treatments in view of historical treatment plans~\cite{chan2014generalized}. It is even believed that the behavior of many biological systems is governed by a principle of optimality with respect to an unknown decision criterion, which can be inferred by tracking the system~\cite{bottasso2006numerica,terekhov2010analytical}. Inverse optimization has also been applied in geoscience~\cite{neumann1984inversion,woodhouse1984mapping}, portfolio selection~\cite{bertsimas2012inverse, iyengar2005inverse}, production planning~\cite{troutt2006behavioral}, inventory management~\cite{carr2000inverse}, network design and control~\cite{ahuja2000faster,burton1992instance,hochbaum2003efficient,farago2003inverse} {and the analysis of electricity prices \cite{ref:GallMor-16}.}

The main thrust of the early literature on inverse optimization is to identify an objective function that explains a {\em single} observation. In the seminal paper~\cite{ahuja2001inverse} the agent solves a static (non-parametric) linear program and reveals her optimal decision to the observer, who then identifies the objective function closest to a prescribed nominal objective, under which the observed decision is optimal. This model was later extended to conic programs~\cite{iyengar2005inverse}, integer programs~\cite{ahmed2005inverse,heuberger2004inverse,schaefer2009inverse,wang2009cutting} and  linearly constrained separable convex programs~\cite{zhang2010inverse}. Another variant of this problem is considered in~\cite{ahmed2005inverse}, where the observer identifies an admissible objective function for which the optimal value of the agent's problem is closest to the observed optimal value corresponding to the unknown true objective.

{This paper focuses on {\em data-driven} inverse optimization problems where the agent solves a parametric optimization problem {\em several times}. Accordingly, the observer has access to a finite sequence of signals and corresponding optimal responses. Using this training data, the observer aims to infer an objective function that accurately predicts the agent's optimal responses to unseen future signals. As in classical regression, this learning task could be addressed by minimizing an empirical loss that penalizes the mismatch between the predicted and true optimal responses to a given signal. {\em Data-driven} inverse optimization problems of this type have only just started to attract attention, and to the best of our knowledge there are currently only three papers that study such problems. In~\cite{keshavarz2011imputing} the observer seeks an objective function under which all observed decisions solve the Karush-Kuhn-Tucker (KKT) optimality conditions of the agent's convex optimization problem. To this end, the observer minimizes some norm of the KKT residuals at all observations. A similar goal is pursued in~\cite{bertsimas2014data}, where the optimality conditions are expressed via variational inequalities that can be reformulated as tractable conic constraints using ideas from robust optimization. This approach has the additional benefit that it extends to more general inverse equilibrium problems, which indicates that inverse optimization problems constitute special instances of mathematical programs with equilibrium constraints. A comprehensive survey of variational inequalities and mathematical programs with equilibrium constraints is provided in~\cite{Harker1990}. The third paper suggests to minimize the empirical average of the squared Euclidean distances between the predicted and true observed decisions, in which case the data-driven inverse optimization problem reduces to a bilevel program~\cite{aswani2015inverse}. 

In summary, all existing approaches to data-driven inverse optimization solve an empirical loss minimization problem over some search space of candidate objectives. Different approaches mainly differ with respect to the loss functions that capture the mismatch between predictions and observations. The {\em KKT loss} used in~\cite{keshavarz2011imputing} quantifies the extent to which the observed response to some signal violates the KKT conditions for a fixed candidate objective. Similarly, the {\em first-order loss} used in \cite{bertsimas2014data} measures the extent to which an observed response violates the first-order optimality conditions. Moreover, the {\em predictability loss} used in~\cite{aswani2015inverse} captures the squared distance between an observed response and the response predicted by a given candidate objective. 

In this paper we introduce the new {\em suboptimality loss}, which quantifies the degree of suboptimality of an observed response under a given candidate objective. We highlight that the predictability and suboptimality losses both enjoy a direct physical meaning, while the KKT and first-order losses are not as easily interpretable. 

Computational experiments in~\cite{keshavarz2011imputing} and~\cite{bertsimas2014data} suggest that empirical loss minimization problems under perfect information are likely to correctly identify the agent's true objective function if there is sufficient training data and the search space of candidate objectives is not too large. In any realistic setting, however, the observer is confronted with imperfect information such as {\em model uncertainty} (the agent's true objective is not one of the candidate objectives), {\em noisy measurements} (the observed signals and responses are corrupted by measurement errors) or {\em bounded rationality} (the agent settles for suboptimal responses due to cognitive or computational limitations). Due to overfitting effects, imperfect information can severely impair the predictive power of a candidate objective obtained via empirical loss minimization. This is simply a manifestation of the notorious `garbage in-garbage out' phenomenon. As imperfect information certainly reflects the norm rather than the exception in inverse optimization, we propose here a systematic approach to combat overfitting via distributionally robust optimization. Specifically, inspired by~\cite{ref:MohKun-14} and~\cite{ref:ShaMohKuh-15}, we use the Wasserstein distance to construct a ball in the space of all signal-response-distributions centered at the empirical distribution on the training samples, and we formulate a distributionally robust inverse optimization problem that minimizes the worst-case risk of loss for any combination of a risk measure with a loss function, where the worst case is taken over all distributions in the Wasserstein ball. If the radius of the Wasserstein ball is chosen judiciously, we can guarantee that it contains the unknown data-generating distribution with high confidence, which in turn allows us to derive rigorous out-of-sample guarantees for the risk of loss of unseen future observations. The proposed distributionally robust inverse optimization problem can naturally be interpreted as a regularization of the corresponding empirical loss minimization problem. While regularization is known to improve the out-of-sample performance of numerous estimators in statistics, it has not yet been investigated systematically in the context of data-driven inverse optimization.

We highlight the following main contributions of this paper relative to the existing literature:

\begin{enumerate}[label=$\bullet$, itemsep = 1mm, topsep = 1mm]
	\item We propose the suboptimality loss as an alternative to the KKT, first-order and predictability losses. The suboptimality loss admits a direct physical interpretation (like the predictability loss) and leads to convex empirical loss minimization problems (like the KKT and first-order losses) whenever the candidate objective functions admit a linear parameterization. In contrast, empirical predictability loss minimzation problems constitute NP-hard bilevel programs even for linear candidate objectives. We also propose the bounded rationality loss, which generalizes the suboptimality loss to situations where the agent is known to select $\delta$-suboptimal decision due to bounded rationality.
	
	\item We leverage the data-driven distributionally robust optimization scheme with Wasserstein balls developed in~\cite{ref:MohKun-14} to regularize empirical inverse optimization problems under imperfect information. As such, the proposed approach offers out-of-sample guarantees for any combination of risk measures and loss functions. In contrast, \cite{bertsimas2014data} develops out-of-sample guarantees only for the value-at-risk of the first-order loss, while~\cite{keshavarz2011imputing} and~\cite{aswani2015inverse} discuss no (finite) out-of-sample guarantees at all.
	
	\item We study the tractability properties of the distributionally robust inverse optimization problem that minimizes the conditional value-at-risk of the suboptimality loss. We prove that this problem is equivalent to a convex program when the search space consists of all linear functions. We also show that this problem admits a safe convex approximation when the search space consists of all convex quadratic functions or all conic combinations of finitely many convex basis functions.
	
	\item We argue that the first-order and suboptimality losses can be used as tractable approximations for the intractable predictability loss, which has desirable statistical consistency properties~\cite{aswani2015inverse} and is the preferred loss function if the observer aims for prediction accuracy. We show that if the candidate objective functions are strongly convex, then the estimators obtained from minimizing the first-order and suboptimality losses admit out-of-sample predictability guarantees. Moreover, the predictability guarantee corresponding to the suboptimality loss is stronger than the one obtained from the first-order loss. Recall that the predictability loss itself cannot be minimized in polynomial time.
	
	\item We show through extensive numerical tests that the proposed distributionally robust approach to inverse optimization attains often better (lower) out-of-sample suboptimality and predictability than the state-of-the art approaches in~\cite{bertsimas2014data} and~\cite{aswani2015inverse}. All of our experiments are reproducible, and the underlying source codes are available at \url{https://github.com/sorooshafiee/InverseOptimization}. 
\end{enumerate}

The rest of the paper develops as follows. In Sections~\ref{sec:prob} and~\ref{sec:imperfect} we formalizes the inverse optimization problem under perfect and imperfect information, respectively. Section~\ref{sec:dro} then introduces the distributionally robust approach to inverse optimization, while Sections~\ref{sec:lin} and~\ref{sec:quad} derive tractable reformulations and safe approximations for distributionally robust inverse optimization problems over search spaces of linear and quadratic candidate objectives, respectively. Numerical results are reported in Section~\ref{sec:sim}.

}


\paragraph{\bf Notation} 
The inner product of two vectors $s,t\in\R^m$ is denoted by $\inner{s}{t} \Let s\tr t$, and the dual of a norm $\|\cdot\|$ on $\R^m$ is defined through $\|t\|_* \Let \sup_{\|s\|\le 1} \inner{t}{s}$. The dual of a proper (closed, solid, pointed) convex cone $\mathcal C\subseteq \R^m$ is defined as $\mathcal C^*\Let\{t\in\R^m:\inner{t}{s}\geq 0\;\forall s\in\mathcal C\}$, and the relation $s\gec{\mathcal C}t$ is interpreted as $s-t\in\mathcal{C}$. Similarly, for two symmetric matrices $Q,R\in\R^{m\times m}$ the relation $Q\succeq R$ ($Q\preceq R$) means that $Q-R$ is positive (negative) semidefinite. The identity matrix is denoted by $\mathbb I$. 
We denote by $\dir{\xi}$ the Dirac distribution concentrating unit mass at $\xi\in\Xi$. The $N$-fold product of a distribution $\PP$ on $\Xi$ is denoted by $\PP^N$, which represents a distribution on the Cartesian product~$\Xi^N$.

\section{Inverse Optimization under Perfect Information}
\label{sec:prob}
{Consider an agent who first receives a random signal $s\in \s  \subseteq \R^m$ and then solves the following parametric optimization problem:
\begin{align}
\label{opt_basic}
\mathop{\text{minimize}}_{x \in \X(s)}~ F(s, x).
\end{align}
Note that both the objective function $F:\R^m \times \R^n \ra \R$ as well as the (multivalued) feasible set mapping $\X:\R^m\rightrightarrows \R^n$ depend on the signal. We assume that the set of minimizers $\X^\star(s)\Let\arg\min_{x\in\X(s)} F(s,x)$ is non-empty for every $s\in \S$. Consider also an independent observer who monitors the signal~$s\in\S$ as well as the agent's optimal response~$x\in\X^\star(s)$. We assume that the observer is ignorant of the agent's preferences encoded by the objective function $F$. Thus, {\em a priori}, the observer cannot predict the agent's response~$x$ to a particular signal~$s$. Throughout the paper we assume that the observed signal-response pairs~$\xi\Let(s,x)$ are governed by some probability distribution $\PP$ supported on $\Xi \Let \big\{(s,x) : s \in \S, ~ x \in \X(s)\big\}$, which can be viewed as the graph of the feasible set mapping~$\X$. Note that the marginal distribution of $s$ under $\mathbb P$ captures the frequency of the exogenous signals, while the conditional distribution of $x$ given $s$ places all probability mass on the argmin set~$\X^\star(s)$. Note that, unless $\X^\star(s)$ is a singleton, the exact conditional distribution of $x$ given $s$ depends on the specific optimization algorithm used by the agent. 

In the following we assume that the observer has access to $N$ independent samples $\xid_i\Let (\sd_i,\xd_i)$ from~$\PP$, which can be used to learn the agent's objective function. As the space of all possible objective functions is vast, the observer seeks to approximate $F$ by some candidate objective function from within a parametric {\em hypothesis space} $\F=\{F_\theta : \theta \in \Theta\}$, where $\Theta$ represents a finite-dimensional parameter set. Ideally, the observer would aim to identify the hypothesis $F_\theta$ closest to $F$, {\em e.g.}, by solving the least squares problem 
\begin{align}
\label{empirical-loss}
\Mini_{\theta \in \Theta} ~ \frac{1}{N} \sum_{i=1}^N \ell_\theta(\sd_i,\xd_i),
\end{align}
where $\ell_\theta(\sd_i,\xd_i)$ denotes the identifiability loss as per the following definition.

\begin{Def}[Identifiability loss] The identifiability loss of model $\theta$ is given by
		\begin{align}
		\label{loss_identifiability} 
		\ell_\theta(s,x) \Let |F(s,x)-F_\theta(s,x)|^2.
		\end{align}
\end{Def}

Unfortunately, the identifiability loss of a training sample cannot be evaluated unless the agent's objective function $F$ is known. Indeed, the observer is blind to the agent's objective values $F(\sd_i,\xd_i)$ and only sees the signals~$\sd_i$ and responses~$\xd_i$. Thus, the identifiability loss cannot be used to learn $F$. It can merely be used to assess the quality of a hypothesis $F_\theta$ obtained with another method in a synthetic experiment where the true objective $F$ is known. As two objective functions have the same minimizers whenever they are related through a strictly monotonically increasing transformation, however, it is indeed fundamentally impossible to learn $F$ from the available training data. At best we can learn the set of its minimizers~$\X^\star(s)$ for every $s$.

If a hypothesis $F_\theta$ is used in lieu of $F$, it can be used to predict the agent's optimal response to a signal $s$ by solving a variant of problem~\eqref{opt_basic}, where $F$ is replaced with $F_\theta$. In the following we define $\X^\star_\theta(s)\Let \arg\min_{y\in\X(s)} F_{\theta}(s,y)$ and refer to any $x\in\X^\star_\theta(s)$ as a response to $s$ predicted by $\theta$. Note that $x\in\X^\star_\theta(s)$ if and only if the response $x$ to $s$ can be explained by model $\theta$. In order to assess the quality of a candidate model $\theta$, the observer should now check whether $\X^\star_\theta(s)\approx\X^\star(s)$ with high probability over $s\in\S$. This can be achieved by solving an empirical loss minimization problem of the form~\eqref{empirical-loss} with a loss function that satisfies $\ell_\theta(s,x) = 0$ if $x\in\X^\star_\theta(s)$ and $\ell_\theta(s,x) > 0$ otherwise. Thus, the loss should vanish if and only if the decision $x$ can be explained as an optimal response to $s$ under model $\theta$. 

In order to learn $\X^\star(s)$, an intuitive approach is to minimize the predictability loss defined below.

\begin{subequations}
\label{loss-examples}
\begin{Def}[Predictability loss] The predictability loss of model $\theta$ is given by
		\begin{align}
		\label{loss_predictability}
		\ell_\theta(s,x) \Let \min_{y\in \X_\theta(s)}~\|x - y \|^2_2.
		\end{align}
		It quantifies the squared Euclidean distance of $x$ from the set of responses to $s$ predicted by $\theta$.
\end{Def}

The predictability loss is known to offer strong statistical consistency guarantees but renders~\eqref{empirical-loss} an NP-hard bilevel optimization problem even if the agent's subproblem is convex~\cite{aswani2015inverse}. Thus, the predictability loss can only be used for low-dimensional problems involving moderate sample sizes. An alternative choice is to minimize the suboptimality loss proposed in this paper, which we will show to be computationally attractive.

\begin{Def}[Suboptimality loss] The suboptimality loss of model $\theta$ is given by
	\label{ell}
		\begin{align}
		\label{loss}
		\ell_\theta(s,x) \Let F_\theta(s,x) - \min_{y \in \X(s)}F_\theta(s,y).
		\end{align}
		It quantifies the suboptimality of $x$ with respect to $F_\theta$ given the signal $s$.
\end{Def}

Another computationally attractive loss function is the degree of violation of the agent's first-order optimality condition~\cite{bertsimas2014data}.

\begin{Def}[First-order loss] If $F_\theta$ is differentiable with respect to $x$, the first-order loss is given by
		\begin{align}
		\label{loss_bertsimas}
		\ell_\theta(s,x) \Let \max_{y \in \X(s) }\inner{\nabla_x F_\theta(s,x)}{x - y}.
		\end{align}
		It quantifies the extent to which $x$ violates the first-order optimality condition of the optimization problem~\eqref{opt_basic} for a given~$s$, where $F$ is replaced with $F_\theta$. Note that the first-order loss vanishes whenever $x$ represents a local minimizer of $F_\theta(s,\cdot)$ over $\X(s)$.
\end{Def}
\end{subequations}

Note that the predictability loss best captures the observer's objective to predict the agent's decisions. However, the suboptimality loss and the first-order loss have better computational properties. Indeed, we will argue below that the learning model~\eqref{empirical-loss} with the loss functions~\eqref{loss} or~\eqref{loss_bertsimas} is computationally tractable under suitable convexity assumptions about the agent's decision problem~\eqref{opt_basic}, the support set $\Xi$ and the hypothesis space $\F$. Thus, we encounter a similar situation as in binary classification, where it is preferable to minimize the convex hinge loss instead of the discontinuous 0-1 loss, which is the actual quantity of interest.

The following proposition establishes basic properties of the loss functions~\eqref{loss-examples}.

\begin{Prop}[Dominance relations between loss functions]
\label{prop:dominance}
Assume that $F_\theta(s,x)$ is convex and differentiable in~$x$, and define $\gamma \geq0$ as the largest number satisfying the inequality
\begin{align}
\label{convex}
 F_\theta(s,y) -  F_\theta(s,x) \ge \inner{\nabla_x F_\theta(s,x)}{y - x} + {\gamma \over 2}\|y-x\|^2 \qquad \forall x,y\in \X(s),~\forall s\in\S.
\end{align}
If $\ell^{\rm p}_\theta$, $\ell^{\rm s}_\theta$ and $\ell^{\rm f}_\theta$ denote the predictability, suboptimality and first-order losses, respectively, then we have
	\begin{align}
		\label{dominance}
		\ell_\theta^{\rm f}(s,x)~ \ge~ \ell^{\rm s}_\theta(s,x) ~\ge~ {\gamma \over 2} \,\ell_\theta^{\rm p}(s,x) \qquad \forall s \in \s,~x \in \X(s).
	\end{align}
	Moreover, all three loss functions are non-negative and evaluate to zero if and only if $x\in \X_\theta(s)$.
\end{Prop}
Note that~\eqref{convex} always holds for $\gamma=0$ due to the first-order condition of convexity \cite[Section~3.1.3]{ref:Boyd}.
\begin{proof}[Proof of Proposition~\ref{prop:dominance}]
Setting $\gamma =0$ and minimizing both sides of~\eqref{convex} over~$y \in \X(s)$ yields $\ell_\theta^{\rm f} (s,x)\ge \ell^{\rm s}_\theta(s,x)$. Next, the first-order optimality condition of the convex program~\eqref{opt_basic} with objective function~$F_\theta$ requires that
\begin{align}
\label{first-opt}
	\inner{\nabla_x F_\theta(s,x)}{y - x} \ge 0 \qquad \forall y \in \X(s)
\end{align}
at any optimal point~$x \in \X^\star_\theta(s)$. Combining the inequalities \eqref{convex} and  \eqref{first-opt} then yields
\begin{align*}
F_\theta(s,y) -  F_\theta(s,x) \ge {\gamma \over 2}\|y-x\|^2 \qquad \forall x \in \X^\star_\theta(s),~y \in \X(s).
\end{align*}	
Minimizing both sides of the above inequality over $x \in \X^\star_\theta(s)$ yields $\ell^{\rm s}_\theta(s,x) \ge {\gamma \over 2}\ell^{\rm p}_\theta(s,x)$. Note that this inequality is only useful for $\gamma>0$, in which case $\X^\star_\theta(s)$ is in fact a singleton. It is straightforward to verify that all loss functions~\eqref{loss-examples} are non-negative and evaluate to zero if and only if $x\in \X^\star_\theta(s)$. In the case of the first-order loss, for instance, this equivalence holds because the first-order condition~\eqref{first-opt} is both necessary and sufficient for the optimality of~$x$. We remark that~\eqref{dominance} remains valid if $\gamma$ depends on $s$ and $\theta$.
\end{proof}

While the basic estimation models in statistical learning all minimize an empirical loss as in~\eqref{empirical-loss}, some inverse optimization models proposed in~\cite{bertsimas2014data} implicitly minimize the worst-case loss across all training samples. To capture both approaches in a unified model, we suggest here to minimize a normalized, positive homogeneous and monotone risk measure $\rho$ that penalizes positive losses. More precisely, we denote by $\rho^\mathbb Q(\ell_\theta)$ the risk of the loss $\ell_\theta(\xi)$ if $\xi=(s,x)$ follows the distribution $\mathbb Q$. The inverse optimization problem~\eqref{empirical-loss} thus generalizes~to
\begin{align}
	\label{saa}
	\mathop{\text{minimize}}_{\theta\in\Theta} ~ \risk^\Pem(\ell_\theta), \qquad\text{where}\qquad \Pem \Let {1 \over N}\sum_{i=1}^N \dir{\xid_i}
\end{align}
represents the {\em empirical distribution} on the training samples. In the remainder we refer to $\risk^\Pem(\ell_\theta)$ as the {\em empirical} or {\em in-sample risk}. Note that~\eqref{saa} reduces indeed to~\eqref{empirical-loss} if we choose the expected value as the risk measure. Two alternative risk measures that could be used in~\eqref{saa} are described in the following example. 

\begin{Ex}[Risk measures] A popular risk measure that the observer could use to quantify the risk of a positive loss is the {\em conditional value-at-risk} ($\cvar$) at level $\alpha\in (0,1]$, which is defined as
\begin{subequations}
	\begin{align}
	\label{cvar}
	\cvar_\alpha^\Q (\ell_\theta) = 
	\inf_{\tau}~\tau + {1 \over \alpha} \EE^\Q\big[\max\{\ell_\theta(s,x) - \tau, 0\} \big] ,
	\end{align}
	see~\cite{ref:Rock&Yry-00}. For $\alpha=1$, the $\cvar$ reduces to the expected value, and for $\alpha\downarrow 0$, it converges to the essential supremum of the loss. Alternatively, the observer could use the {\em value-at-risk} ($\var$) at level $\alpha\in [0,1]$ defined~as
	\begin{align}
	\label{var}
	\var_\alpha^\Q (\ell_\theta) = 
	\inf_{\tau}~\left\{ \tau ~:~\Q\big[ \ell_\theta(s,x)\leq \tau\big] \geq 1-\alpha\right\}.
	\end{align}
	Note that the $\var$ coincides with the upper $(1-\alpha)$-quantile of the loss distribution. Moreover, if $\ell_\theta(s,x)$ has a continuous marginal distribution under $\Q$, then $\cvar_\alpha^\Q (\ell_\theta)$ coincides with the expected loss above $\var^\Q_\alpha(\ell_\theta)$. 
\end{subequations}	
\end{Ex}

By definition, the loss function $\ell_\theta(\xi)$ is non-negative for all $\xi\in\Xi$ and $\theta\in\Theta$. The monotonicity and normalization of the risk measure $\rho$ thus imply that
\[
	\risk^\Pem(\ell_\theta)\geq \risk^\Pem(0) =0\qquad \forall \theta\in\Theta,
\]
which in turn implies that the optimal value of problem~\eqref{saa} is necessarily larger than or equal to zero. Moreover, if the agent's true objective function is contained in~$\F$, that is, if $F=F_{\theta\opt}$ for some $\theta\opt\in \Theta$, then the loss~$\ell_{\theta\opt}(\xid_i)$ vanishes for all~$i$, indicating that the optimal value of~\eqref{saa} is zero and that $\theta\opt$ is optimal in~\eqref{saa}. In this case, the optimal {\em value} of the in-sample risk minimization problem~\eqref{saa} is known {\em a priori}, and the only informative output of any solution scheme is an {\em optimizer}, that is, a model $\theta'\in\Theta$ with zero in-sample risk. If the number of training samples is moderate, then there may be multiple optimal solutions, and $\theta'$ may differ from the agent's true model $\theta^\star$. 

\begin{Rem}[Choice of risk measures]
	\label{rem:risk}
	If $F=F_{\theta\opt}$ for $\theta\opt\in\Theta$, then the minimum of~\eqref{saa} vanishes and is minimized by $\theta\opt$ irrespective of $\risk$. Thus, one might believe that the choice of the risk measure is immaterial for the inverse optimization problem. However, different risk measures may result in different solution sets. For example, if $\risk$ is the $\cvar$ at level $\alpha\in (0,1]$, then $\theta'$ is a minimizer of~\eqref{saa} if and only if $\ell_{\theta'}(\sd_i, \xd_i)=0$ for all $i\le N$. In contrast, if $\risk$ is the $\var$ at level $\alpha\in [0,1]$, then $\theta'$ is a minimizer of~\eqref{saa} if and only if $\ell_{\theta'}(\sd_i, \xd_i)=0$ for a portion of at least $1-\alpha$ of all $N$ training samples. Thus, the use of $\var$ may lead to an inflated solution set. Moreover, the choice of $\risk$ impacts the tractability of~\eqref{saa}; see Sections~\ref{sec:lin} and~\ref{sec:quad}.
\end{Rem} 
}

\section{Inverse Optimization under Imperfect Information}
\label{sec:imperfect}

{ The proposed framework for inverse optimization described in Section~\ref{sec:prob} is predicated on the assumption of perfect information. Specifically, it is assumed that the agent is able to determine and implement the best response $x$ to any given signal $s$ and that the observer can measure $s$ and $x$ precisely. Moreover, it is implicitly assumed that the family $\F$ of candidate objective functions contains the agent's true objective function $F$. In practice, however, the observer may be confronted with the following challenges: 

\begin{enumerate}[label=(\ch{\arabic{*}}), itemsep = 2mm, topsep = 1mm]
	\item[(i)]{\bf Model uncertainty:} The hypothesis space $\F$ chosen by the observer may not be rich enough to contain the agent's true objective function $F$ or any strictly increasing transformation of $F$ that encodes the same preferences.
	
	\item[(ii)] {\bf Measurement noise:} The observed signal-response pairs may be corrupted by noise, which prevents the observer from measuring them exactly. 
		
	\item[(iii)] {\bf Bounded rationality:} The agent may settle for a suboptimal decision $x$ due to cognitive or computational limitations. Even if the best response can be computed exactly, an exact implementation of the desired best response may not be possible due to implementation errors \cite{ref:BenElGhaNem09}.
	
\end{enumerate}

In the presence of {\em model uncertainty}, that is, if neither $F$ nor any strictly increasing transformation of~$F$ is contained in the chosen hypothesis space $\F$, then there exists typically no $\theta\in\Theta$ such that the loss functions described in Section~\ref{sec:prob} vanish on all training samples. In this case a perfect recovery of the agent's preferences is fundamentally impossible, and the optimal value of the empirical risk minimization problem~\eqref{saa} is positive. The best the observer can hope for is to learn the parametric hypothesis that most accurately (but imperfectly) captures the agent's true preferences.

In the presence of {\em measurement noise} the observed training samples~$\xid_i=(\sd_i,\xd_i)$ represent random perturbations of some unobservable pure samples~$\wt\xi_i=(\wt s_i,\wt x_i)$. While $\wt x_i$ is an exact optimal response to an unperturbed signal~$\wt s_i$, the noisy response $\xd_i$ generically constitutes a suboptimal---maybe even infeasible---response to $\sd_i$. We will henceforth assume that the noisy samples $\xid_i$ are mutually independent and follow an {\em in-sample distribution} $\PP_{\rm in}$, while the corresponding unperturbed samples $\xid_i$ are governed by an {\em out-of-sample distribution} $\PP_{\rm out}$. While the distribution $\PP_{\rm out}$ of the perfect samples is supported on $\Xi$ by construction, the in-sample distribution~$\PP_{\rm in}$ of the noisy samples may or may not be supported on $\Xi$. If the noisy samples can materialize outside of~$\Xi$, then we call the noise {\em inconsistent}. If all noisy samples are guaranteed to reside within $\Xi$, on the other hand, then we call the noise {\em consistent}. Thus, in the presence of measurement noise, the observer faces the challenging task to learn $\PP_{\rm out}$ from samples of $\PP_{\rm in}$. Note that in the absence of noise we have $\PP_{\rm in}=\PP_{\rm out}$, which coincides with the distribution $\PP$ introduced in Section~\ref{sec:prob}.

Agents suffering from \emph{bounded rationality} may not be able to solve~\eqref{opt_basic} to global optimality. Such agents may respond to a given signal~$s$ with a $\delta$-suboptimal solution~$x_\delta$, that is, a decision $x_\delta\in\X(s)$ with $F(s, x_\delta)\leq \min_{x\in\X(s)}F(s, x)+\delta$. Here, the parameter $\delta\geq 0$ quantifies the agent's irrationality. Indeed, if $\delta>0$, the agent accepts a cost increase of up to $\delta$ for the freedom of choosing a $\delta$-suboptimal decision, which may require fewer cognitive or computational resources than finding a global minimizer. Note that the observer may mistakenly perceive the effects of bounded rationality as (consistent) measurement noise or vice versa. So one might argue that there is no need to distinguish between these two types of imperfect information. However, bounded rationality is fundamentally different from measurement noise if the observer {\em knows} that the imperfect measurements are caused by bounded rationality. If the imperfections originate from measurement noise, then the observer aims to filter out the noise in order to predict the agent's pure decisions. If the imperfections originate from bounded rationality, on the other hand, then the observer aims to predict the agent's $\delta$-suboptimal decisions and not the ideal global optimizers. In other words, the observer always aims to predict the agent's responses bar measurement noise, irrespective of whether these responses are rational or not. An observer who knows the agent's degree of irrationality $\delta$ may thus improve the predictive power of her learning model by replacing the suboptimality loss~\eqref{loss} with the bounded rationality loss.

\begin{Def}[Bounded rationality loss] The bounded rationality loss of model $\theta$ is given by
		\begin{align}
		\label{bounded-ratioinality-loss}
		\ell_\theta(\delta; s,x) \Let \max\big\{F_\theta(s,x) - \min_{y \in \X(s)} F_\theta(s,y)-\delta,0\big\}.
		\end{align}
		It quantifies the amount by which the suboptimality of $x$ with respect to $F_\theta$ given signal $s$ exceeds $\delta\geq 0$. By construction, $\ell_\theta(\delta;s,x)=0$ whenever $x$ is a $\delta$-supoptimal response to the signal~$s$, and $\ell_\theta(s,x)>0$ otherwise. 
\end{Def}

Note that the bounded rationality loss with $\delta=0$ reduces to the usual suboptimality loss~\eqref{loss}. Even if it is known that the agent suffers from bounded rationality, the constant $\delta$ is unlikely to be available in practice. However, as long as there are no other sources of imperfect information (such as model uncertainty or measurement noise), the observer can jointly estimate $\delta$ and $\theta$ from the training data by solving
\begin{equation}
	\label{eq:inverse-opt-bounded-rationality}
	\mathop{\rm minimize}_{\theta\in\Theta, \,\delta\geq 0} \left\{ \delta : \ell_\theta(\delta; \sd_i,\xd_i)=0\; \forall i\le N\right\}.
\end{equation}
This variant of the inverse optimization problem identifies the smallest bounded rationality constant $\delta$ that explains all observed responses as $\delta$-suboptimal decisions under model $\theta$. Note that \eqref{eq:inverse-opt-bounded-rationality} constitutes a convex optimization problem as long as $\Theta$ is convex and the hypotheses $F_\theta(s,x)$ are affinely parameterized in $\theta$, which guarantees that $\ell_\theta(\delta; s,x)$ is jointly convex in $\theta$ and $\delta$ for any fixed $s$ and $x$. Moreover, instead of solving \eqref{eq:inverse-opt-bounded-rationality}, one could equivalently solve the empirical risk minimization problem~\eqref{saa} with the suboptimality loss and the essential supremum risk measure to find $\theta$ and then set $\delta$ to the resulting optimal objective value.

To some extent, all of the complications discussed in this section are inevitable in any real application. In fact, they are likely to reflect the norm rather than the exception. Thus, the main objective of this paper is to develop inverse optimization models that can cope with imperfect information in a principled manner.}

\section{Distributionally Robust Inverse Optimization}
\label{sec:dro}

{By combining different loss functions with different risk measures, one may synthesize different empirical risk minimization problems of the form~\eqref{saa}. By construction, any solution of \eqref{saa} has minimum in-sample risk. However, the in-sample risk merely captures historical performance and is therefore of little practical interest. Instead, the observer seeks models that display promising performance on unseen future data. 

\begin{Def}[Out-of-sample risk]
	\label{def:oos}
	We refer to $\risk^{\PP_{\rm out}}(\ell_{\theta})$ as the out-of-sample risk of the model $\theta\in\Theta$. 
\end{Def}

Ideally, the observer would want to minimize the out-of-sample risk over all candidate models $\theta\in\Theta$. This is impossible, however, because the out-of-sample distribution $\PP_{\rm out}$ of the signal-response pairs is unknown, and only a finite set of training samples from $\PP_{\rm in}$ is available (recall that $\PP_{\rm in}=\PP_{\rm out}$ if measurements are perfect). In this situation, the observer has to settle for a data-driven solution $\widehat\theta_N\in\Theta$ that depends on the training samples and attains---hopefully---a low out-of-sample risk. We emphasize that, due to its dependence on the training samples, $\widehat\theta_N$ constitutes a random object, whose stochastics is governed by the product distribution~$\PP_{\rm in}^N$. A simple data-driven solution is obtained, for instance, by solving the empirical risk minimization problem~\eqref{saa}.

While it is impossible to minimize the out-of-sample risk on the basis of the training samples, it is sometimes possible to establish data-driven out-of-sample guarantees in the sense of the following definition.

\begin{Def}[Out-of-sample guarantee]
	\label{def:cert}
	We say that a data-driven solution $\widehat \theta_N$ enjoys an out-of-sample performance guarantee at significance level $\beta\in[0,1]$ if there exists a data-driven certificate $\widehat J_N$ with
\begin{align}
	\label{oos-guarantee}
	\PP_{\rm in}^N \big[\risk^{\PP_{\rm out}}(\ell_{\widehat \theta_N})\leq \widehat \cert_N \big] \geq 1-\beta.
\end{align}
\end{Def}

Note that the probability in \eqref{oos-guarantee} is evaluated with respect to the distribution of $N$ independent (potentially noisy) training samples, which impact both the data-driven solution $\wh\theta_N$ and the certificate $\wh\cert_N$. Note also that the certificate $\wh\cert_N$ can be viewed as an upper $(1-\beta)$-confidence bound on the out-of-sample risk of $\widehat \theta_N$. Thus, we sometimes refer to the confidence level $1-\beta$ as the certificate's {\em reliability}.

As the ideal goal to minimize the out-of-sample risk is unachievable, the observer might settle for the more modest goal to determine a data-driven solution that admits a low certificate with a high reliablity. We will now argue that this secondary goal is achievable by adopting a distributionally robust approach. Specifically, we will use the $N$ training samples to design an ambiguity set $\amb_N$ that contains the out-of-sample distribution $\PP_{\rm out}$ of the (perfect) signal-response pairs with confidence~$1-\beta$. Next, we construct the data-driven solution $\wh \theta_N$ and the corresponding certificate $\wh \cert_N$ by minimizing the worst-case risk across all models $\theta\in\Theta$, where the worst case is taken with respect to all signal-response distributions in the ambiguity set $\amb_N$, that is, we set
\begin{align}
\label{DRO-inv-opt}
\wh \theta_N \in \argmin_{\theta \in \Theta} ~\sup_{\Q \in \amb_N} ~\risk^{\Q}(\ell_\theta) \qquad\text{and}\qquad 
\wh J_N \Let \min_{\theta \in \Theta}  \sup_{\Q \in \amb_N} ~\risk^{\Q}(\ell_\theta).
\end{align}
It is clear that if $\PP_{\rm in}^N[\PP_{\rm out}\in\amb_N ]\geq 1-\beta$, then the distributionally robust solution $\wh\theta_N$ and the corresponding certificate $\wh\cert_N$ defined above satisfy the out-of-sample guarantee~\eqref{oos-guarantee}. In order to ensure that $\amb_N$ contains the unknown out-of-sample distribution $\PP_{\rm out}$ with confidence $1-\beta$, we construct the ambiguity set $\amb_N$ as a ball in the space of probability distributions with respect to the Wasserstein metric as suggested in~\cite{ref:MohKun-14}.

\begin{Def}[Wasserstein metric]
	\label{def:wass}
	For any integer $p\geq 1$ and closed set $\Xi\subset\R^{n+m}$ we let $\M^p(\Xi)$ be the space of all probability distributions $\Q$ supported on $\Xi$ with $\EE^\Q\big[\|\xi\|^p\big] = \int_\Xi \|\xi\|^p \,\Q(\diff\xi)<\infty$. The $p$-Wasserstein distance between two distributions $\Q_1,\Q_2\in \M^p(\R^{n+m})$ is defined as
	\begin{align*}
	\Wass{p}{\Q_1}{\Q_2} \Let \inf \left\{ \left( \int \| \xi_1 -  \xi_2 \|^p\,  \Pi(\diff \xi_1, \diff \xi_2)\right)^{1/p} : 
	\begin{array}{l}\mbox{$\Pi$ is a joint distribution of $\xi_1$ and $\xi_2$} \\ 
	\mbox{with marginals $\Q_1$ and $\Q_2$, respectively}\! \end{array}\right\}. 
	\end{align*}
\end{Def}

The Wasserstein distance $\Wass{p}{\Q_1}{\Q_2}$ can be viewed as the ($p$-th root of the) minimum cost for moving the distribution $\Q_1$ to $\Q_2$, where the cost of moving a unit mass from $\xi_1$ to $\xi_2$ amounts to~$\| \xi_1 -  \xi_2 \|^p$. The joint distribution $\Pi$ of $\xi_1$ and $\xi_2$ is therefore naturally interpreted as a mass transportation plan. 

We define the ambiguity set as a $p$-Wasserstein ball in $\M^p(\Xi)$ centered at the empirical distribution $\Pem$ defined in \eqref{saa}. Specifically, we define the $p$-Wasserstein ball of radius $\eps$ around $\Pem$ as
\begin{align*}
	\ball{p}{\Pem}{\eps} \Let \left\{ \Q \in \M^p(\Xi) ~:~ \Wass{p}{\Q}{\Pem} \le \eps\right \}.
\end{align*}
Note that if $\eps=0$ and the empirical distribution is supported on $\Xi$, which is necessarily true in the absence of measurement noise, then the Wasserstein ball $\ball{p}{\Pem}{\eps}$ shrinks to the singleton set that contains only the empirical distribution. In this case, the distributionally robust inverse optimization problem~\eqref{DRO-inv-opt} reduces to the empirical risk minimization problem~\eqref{saa}. In order to establish out-of-sample guarantees, we must assume that the $p$-Wasserstein distance between $\PP_{\rm in}$ and $\PP_{\rm out}$ is bounded by a known constant~$\eps_0\geq 0$. This is the case, for instance, if the noise is additive and all noise realizations are bounded by~$\eps_0$ with certainty.

\begin{Prop}[Out-of-sample guarantee]
	\label{prop:out-of-sample} 
	Assume that there exists $a>1$ with $A\Let\EE^{\PP_{\rm in}}[\exp(\|\xi\|^a)]<\infty$. Assume also that $\beta \in (0,1)$ is a prescribed significance level, $\Wass{p}{\PP_{\rm in}}{\PP_{\rm out}}\le \eps_0$, and $m+n \neq 2p$.\footnote{Proposition~\ref{prop:out-of-sample} readily extends to the case $n+m = 2p$ at the expense of additional notation by leveraging \cite[Theorem~2]{ref:FouGui-14}.} Then, there exist constants $c_1,c_2>0$ that depend only on $a$, $A$, $m$ and $n$ such that $\PP_{\rm in}^N[\PP_{\rm out}\in \ball{p}{\Pem}{\eps}]\geq 1-\beta$ whenever $\eps\geq \eps_0+\eps_N(\beta)$, where
	\begin{align}
	\label{eps_N}
	\eps_N(\beta) \Let \left\{ \begin{array}{ll}  \Big({\log (c_1\beta^{-1}) \over c_2N} \Big)^{\min\big\{p(m+n)^{-1} ,\frac{1}{2}\big\}} & \text{if } N \ge {\log(c_1\beta^{-1}) \over c_2}, \vspace{1mm} \\
	\Big({\log (c_1\beta^{-1}) \over c_2N} \Big) & \text{if } N < {\log(c_1\beta^{-1}) \over c_2}.		\end{array}\right.
	\end{align}  
\end{Prop}
\begin{proof} 
	Select any $\eps\geq \eps_0+\eps_N(\beta)$. Then, the triangle inequality implies
	\begin{align*}
		\Wass{p}{\PP_{\rm out}}{\Pem} \leq \Wass{p}{\PP_{\rm out}}{\PP_{\rm in}} +  \Wass{p}{\PP_{\rm in}}{\Pem} .
	\end{align*}
	By assumption, the first term on the right hand side is bounded by $\eps_0$ with certainty. Theorem~3.5 in~\cite{ref:MohKun-14}, which leverages a powerful measure concentration result developed in \cite[Theorem~2]{ref:FouGui-14}, further guarantees that the second term is bounded above by $\eps_N(\beta)$ with confidence $1-\beta$. As $\PP_{\rm out}$ is supported on $\Xi$, we may thus conclude that $\PP_{\rm out}\in \ball{p}{\Pem}{\eps}$ with probability $1-\beta$. This observation completes the proof. 
\end{proof}

Proposition~\ref{prop:out-of-sample} ensures that the distributionally robust solution $\wh\theta_N$ and the corresponding certificate $\wh\cert_N$ induced by a Wasserstein ambiguity set of radius $\eps\geq \eps_0+\eps_N(\beta)$ satisfy the out-of-sample guarantee~\eqref{oos-guarantee}. One can further show that if there is no measurement noise ($\PP_{\rm in}=\PP_{\rm out}=\PP$) while $\beta_N\in(0,1)$ for $N\in\mathbb N$ satisfies $\sum_{N=1}^\infty \beta_N<\infty$ and $\lim_{N\rightarrow\infty} \eps_N(\beta_N)=0$,\footnote{A possible choice is $\beta_N=\exp(-\sqrt{N})$.} then any accumulation point of $\{\wh\theta_N\}_{N\in\mathbb N}$ is $\PP^\infty$-almost surely a minimizer of the the out-of-sample risk $\risk^{\PP}(\ell_\theta)$ over $\theta\in\Theta$; see~\cite[Theorem~3.6]{ref:MohKun-14}.

Besides offering rigorous out-of-sample and asymptotic guarantees, the proposed approach to distributionally robust inverse optimization can be shown to be tractable if the search space contains only linear or quadratic hypotheses, and risk is measured by the $\cvar$ of the suboptimality loss~\eqref{loss} or the first-order loss~\eqref{loss_bertsimas}. Unfortunately, tractability is lost when minimizing the predictability loss~\eqref{loss_predictability}, which is the actual quantitiy of interest. However, the computable distributionally robust solutions $(\wh\theta_N^{\rm s},\wh\cert_N^{\rm s})$ and $(\wh\theta_N^{\rm f},\wh\cert_N^{\rm f})$ corresponding to the suboptimality and first-order losses, respectively, can be used to construct out-of-sample guarantees for the predictability loss if the hypotheses are uniformly strongly convex with parameter~$\gamma>0$. Indeed, if $\ell^{\rm p}_\theta$, $\ell^{\rm s}_\theta$ and $\ell^{\rm f}_\theta$ denote the predictability, suboptimality and first-order losses, respectively, we have
\begin{subequations}
\begin{align}
	\label{oos-predictability-guarantee1}
	\PP_{\rm in}^N \left[\risk^{\PP_{\rm out}}(\ell^{\rm s}_{\widehat \theta^{\rm s}_N})\leq \widehat \cert_N^{\rm s} \right] \geq 1-\beta \qquad\implies \qquad
	\PP_{\rm in}^N \left[\risk^{\PP_{\rm out}}(\ell^{\rm p}_{\widehat \theta^{\rm s}_N})\leq \frac{2}{\gamma}\widehat \cert^{\rm s}_N \right] \geq 1-\beta
\end{align}
and
\begin{align}
	\label{oos-predictability-guarantee2}
	\PP_{\rm in}^N \left[\risk^{\PP_{\rm out}}(\ell^{\rm f}_{\widehat \theta^{\rm f}_N})\leq \widehat \cert^{\rm f}_N \right] \geq 1-\beta \qquad\implies \qquad
	\PP_{\rm in}^N \left[\risk^{\PP_{\rm out}}(\ell^{\rm p}_{\widehat \theta^{\rm f}_N})\leq \frac{2}{\gamma}\widehat \cert^{\rm f}_N \right] \geq 1-\beta,
\end{align}
\end{subequations}
where both implications follow from the dominance relation~\eqref{dominance} established in Proposition~\ref{prop:dominance} and the scale invariance and monotonicity of the $\cvar$. Thus, both $\widehat \theta^{\rm s}_N$ and $\widehat \theta^{\rm f}_N$ are efficiently computable and offer an out-of-sample guarantee for the predictability loss. While both guarantees involve the same confidence level $1-\beta$, however, the underlying certificates $\cert^{\rm s}_N$ and $\cert^{\rm f}_N$ are generically different. One can again use the dominance relation~\eqref{dominance} from Proposition~\ref{prop:dominance} and the monotonicity of the $\cvar$ to show that $\cert^{\rm s}_N\le \cert^{\rm f}_N$. Thus, minimizing the suboptimality loss results in a weakly stronger predictability guarantee. This reasoning suggests that the observer should favor the suboptimality loss~\eqref{loss} over the first-order loss~\eqref{loss_bertsimas}. We emphasize that unlike the suboptimality loss, however, the first-order loss leads to tractable inverse optimization models even in the presence of {\em several} strategically interacting agents~\cite{bertsimas2014data}.

\begin{Rem}[Out-of-sample guarantees in~\cite{bertsimas2014data}]
The out-of-sample guarantees provided in~\cite{bertsimas2014data} only apply to the $\var$, and an extension to other risk measures is not envisaged. Specifically, \cite[Theorem~6]{bertsimas2014data} provides an out-of-sample guarantee for the $\var$ of the first-order loss, while \cite[Theorem~6]{bertsimas2014data} offers an out-of-sample guarantee for the $\var$ of the predictability loss. In contrast, the distributionally robust approach discussed here offers out-of-sample guarantees for any normalized, positive homogeneous and monotone risk measure including the $\var$ or any coherent risk measure such as the $\cvar$ etc. 
\end{Rem}
}

\section{Linear Hypotheses}
\label{sec:lin}
On the one hand, the {hypothesis space $\F$} should be rich enough to contain the agent's unknown true objective function $F$. On the other hand, $\F$ should be small enough to ensure tractability of the distributionally robust inverse optimization problem~\eqref{DRO-inv-opt} and to prevent degeneracy of its optimal solutions. A particular class $\F$ that strikes this delicate balance and proves useful in many applications is the family of linear objective functions $F_\theta(s,x) \Let \inner{\theta}{x}$. {The corresponding search space $\Theta\subseteq \mathbb R^n$ may account for prior information on the agent's objective and should not contain $\theta=0$, which corresponds to a trivial constant objective function that renders every response optimal. Examples of tractable search spaces are listed below.

\begin{Ex}[Tractable search spaces]
	\label{Ex:Theta}
	If there is no prior information on $F$, it is natural to set
	\begin{subequations}
	\label{eq:search-spaces-linear}
	\begin{align}
	\label{eq:search-space}
		\Theta \Let \big\{\theta \in \R^n : \|\theta\|_\infty = 1 \big\}. 
	\end{align}
	Note that the normalization $\|\theta\|_\infty = 1$ is non-restrictive because the objective functions corresponding to $\theta$ and $\kappa\,\theta$ imply the same preferences for any model $\theta\neq 0$ and scaling factor $\kappa>0$. In fact, we could define~$\Theta$ as the unit sphere induced by any norm on $\R^n$. However, the  $\infty$-norm stands out from a computational perspective. While all norm spheres are non-convex and therefore {\em a priori} unattractive as search spaces, the $\infty$-norm sphere decomposes into $2n$ polytopes---one for each facet. This polyhedral decomposition property allows us to optimize efficiently over $\Theta$. 
	
	If $F$ is known to be non-decreasing in the agent's decisions, a natural choice is
	\begin{align}
	\label{eq:search-space2}
		\Theta\Let\{\theta \in \R^n: \|\theta\|_1=1,\theta\geq 0\}.
	\end{align}
	This search space has been used in \cite{keshavarz2011imputing} and constitutes a single convex polytope. 
	
	If $F$ is believed to reside in the vicinity of a nominal objective function $\inner{\theta_0}{x}$ as in \cite{ahuja2001inverse}, then we may set  
	\begin{align}
	\label{eq:search-space3}
		\Theta \Let \big\{\theta \in \R^n : \|\theta - \theta_0\| \le \Gamma \big\},
	\end{align}
	\end{subequations}	
	where $\|\cdot\|$ denotes a generic norm, and $\Gamma$ reflects the degree of uncertainty about the nominal model~$\theta_0$. 
\end{Ex}
}

{When focusing on linear hypotheses, the {suboptimality} loss function~\eqref{loss} reduces to
\begin{align}
\label{loss-lin}
	F_\theta(s,x)-\min_{y\in\X(s)} F_\theta(s,y) = \inner{\theta}{x}-\min_{y\in\X(s)}\inner{\theta}{y} = \max_{y\in\X(s)} \inner{\theta}{x-y}= \max_{y\in\X(s)} \inner{\nabla_x F_\theta(s,x)}{x-y}
\end{align}
and thus equals the first-order loss~\eqref{loss_bertsimas}, }which is positive homogeneous and subadditive in $\theta$. The tractability results to be established below rely on the following assumption. 

{
\begin{As}[Conic representable support]
	\label{As:conic}
	The signal space~$\S$ and the feasible set~$\X(s)$ are conic representable, that is,
	\begin{align*}
	\mathbb S = \big\{s \in \R^m :  Cs  \gec{\coneXi} d \big\}\qquad\text{and} \qquad \X(s)  = \{x \in \R^n : Wx \gec{\coneX} Hs+ h \}\quad \forall s\in \mathbb S, 
	\end{align*}
	where the relations `$\gec{\coneXi}$' and `$\gec{\coneX}$' represent conic inequalities with respect to some proper convex cones $\coneXi$ and $\coneX$ of appropriate dimensions, respectively. The set $\Xi$ of all possible signal-response pairs thus reduces to
	\begin{align*}
	\Xi = \big\{(s,x) \in \R^m\times\R^n:  Cs  \gec{\coneXi} d,~ Wx \gec{\coneX} Hs+ h\big\}.
	\end{align*}
	We also assume that the convex set $\Xi$ admits a Slater point.
\end{As}

Under Assumption~\ref{As:conic}, the suboptimality loss $\ell_\theta(s,x)$ is concave in $(s,x)$ for every fixed $\theta$, see, {\em e.g.},~\cite[Section~3.2.5]{ref:Boyd}. We are now ready to state our first tractability result for the class of linear hypotheses}.

\begin{Thm}[Linear {hypotheses} and suboptimality loss]
	\label{thm:lin}
	Assume that $\F$ represents the class of linear {hypotheses with a search space of the form~\eqref{eq:search-spaces-linear} and that Assumption~\ref{As:conic} holds.} If the observer uses the suboptimality loss~\eqref{loss} and measures risk using the $\cvar$ at level $\alpha\in(0,1]$, then  the distributionally robust inverse optimization problem~\eqref{DRO-inv-opt} over the $1$-Wasserstein ball is equivalent to the finite conic program\footnote{Strictly speaking, if $\Theta$ is an $\infty$-norm ball of the form~\eqref{eq:search-space}, then \eqref{cvar-lin} can be viewed as a family of $2n$ finite conic programs because $\Theta$ is non-convex but decomposes into $2n$ convex polytopes.}
	\begin{align}
	\label{cvar-lin}
	\begin{array}{llll} \text{\em minimize} & \displaystyle \tau+\frac{1}{\alpha}\left({\eps} \lambda  + {1 \over N}\sum\limits_{i = 1}^{N} r_i\right) \vspace{1mm}\\
	\text{\em subject to} & 	\theta\in\Theta,\;\;	\lambda \in\mathbb R_+,\;\;\tau,r_i\in\R,\;\; \phi_{i1},\phi_{i2}\in \coneXi^*, \;\;\mu_{i1},\mu_{i2},\gamma_{i} \in \coneX^* &\forall i \le N\\ 
	& \inner{C\sd_i-d}{\phi_{i1}} + \inner{W\xd_i-H\sd_i-h}{\mu_{i1}+\gamma_{i}} \leq r_i + \tau&\forall i \le N \vspace{1mm}\\	
	& \inner{C\sd_i-d}{\phi_{i2}} + \inner{W\xd_i-H\sd_i-h}{\mu_{i2}} \leq r_i,\;\; \theta=W\tr\gamma_{i}&\forall i\le N\vspace{1mm}\\
	& \left\|\begin{pmatrix}C\tr\phi_{i1} - H\tr(\mu_{i1} + \gamma_{i}) \\ W\tr(\mu_{i1} + \gamma_{i})\end{pmatrix} \right\|_* \le \lambda,\;\; \left\|\begin{pmatrix}C\tr\phi_{i2} - H\tr\mu_{i2} \\ W\tr\mu_{i2}\end{pmatrix} \right\|_* \le \lambda&\forall i \le N.
	\end{array}
	\end{align}
\end{Thm}

{We emphasize that Theorem~\ref{thm:lin} remains valid if the training samples are inconsistent with the given support information, that is, if $(\sd_i, \xd_i)\notin\Xi$ for some $i\le N$, in which case $\ell_\theta(\sd_i, \xd_i)$ can even be negative.}

\begin{proof}[Proof of Theorem~\ref{thm:lin}]
	By the definition of $\cvar$, the objective function of~\eqref{DRO-inv-opt} can be expressed as
	\begin{align}
	\label{wc_cvar}
	\sup_{\Q \in \ball{1}{\Pem}{\eps}} \risk^{\Q}(\ell_\theta) & = \sup_{\Q \in \ball{1}{\Pem}{\eps}}  \inf_{\tau} \tau + {1 \over \alpha} \EE^\Q\big[\max\{\ell_\theta(s,x) - \tau, 0\} \big] \\
	& = \inf_{\tau} \tau + {1 \over \alpha} \sup_{\Q \in \ball{1}{\Pem}{\eps}}  \EE^\Q\big[\max\{\ell_\theta(s,x) - \tau, 0\} \big]. \nonumber
	\end{align}
	The interchange of the maximization over $\Q$ and the minimization over $\tau$ in the second line is justified by Sion's minimax theorem~\cite{ref:Sion-58}, {which applies because the Wasserstein ball $\ball{1}{\Pem}{\eps}$ is weakly compact \cite[p.\ 2298]{ref:Bois-11}.}
	The subordinate worst-case expectation problem in the second line of~\eqref{wc_cvar} constitutes a semi-infinite linear program. As the corresponding integrand is given by the maximum of $\ell_\theta(s,x)-\tau$ and $0$, both of which can be viewed as proper concave functions in $(s,x)$, this worst-case expectation problem admits a strong dual semi-infinite linear program of the form
	\begin{align}
	\label{semiinf-linear}
	\begin{array}{clll} \Inf{{\lambda\geq 0,r_i}} &  {\eps} \lambda  + {1 \over N}\sum\limits_{i = 1}^{N} r_i \vspace{1mm}\\
	\st &  \Sup{(s,x)\in\Xi} \Sup{y \in \X(s)} \inner{\theta}{x - y} -\tau-\lambda\|(s,x)-(\sd_i,\xd_i)\|\leq r_i&\forall i\leq N  \vspace{1mm}\\
	&  \Sup{(s,x)\in\Xi} -\lambda\|(s,x)-(\sd_i,\xd_i)\|\leq r_i&\forall i\leq N,
	\end{array}
	\end{align}
	see Theorem~4.2 in~\cite{ref:MohKun-14} for a detailed derivation of~\eqref{semiinf-linear} for more general integrands. By the definitions of $\X(s)$ and $\Xi$ put forth in Assumption~\ref{As:conic}, respectively, the $i$-th member of the first constraint group in~\eqref{semiinf-linear} holds if and only if the optimal value of the conic program 
	\begin{align*}
	&\begin{array}{clll} \Sup{{s,x,y}} & \inner{\theta}{x - y} -\tau-\lambda\|(s,x)-(\sd_i,\xd_i)\|\vspace{1mm}\\
	\st &   Cs \gec{\coneXi} d,\; Wx \gec{\coneX} Hs+ h,\; Wy \gec{\coneX} Hs + h
	\end{array}
	\end{align*}
	is smaller or equal to $r_i$. The dual of this conic program is given by
	\begin{align*}
	&\begin{array}{clll} \Inf{} &  \inner{C\sd_i-d}{\phi_{i1}} + \inner{W\xd_i-H\sd_i-h}{\mu_{i1}+\gamma_{i}} -\tau \vspace{1mm}\\
	\st & \phi_{i1} \in \coneXi^*, \; \mu_{i1}, \gamma_i \in \coneX^* \vspace{1mm} \\
	& \left\|\begin{pmatrix}C\tr\phi_{i1} - H\tr(\mu_{i1} + \gamma_i) \\ W\tr(\mu_{i1} + \gamma_i)\end{pmatrix} \right\|_* \le \lambda,\quad\theta=W\tr\gamma_{i} ,
	\end{array}
	\end{align*}
	and strong duality holds because the uncertainty set $\Xi$ {contains a Slater point due to Assumption~\ref{As:conic}}. Thus, the $i$-th member of the first constraint group in~\eqref{semiinf-linear} holds if and only if there exist $\phi_{i1} \in \coneXi^*$ and $\mu_{i1}, \gamma_{i} \in \coneX^*$ such that  $\theta=W\tr\gamma_i$, 
	\begin{align*}
	\begin{array}{clll} &  \inner{C\sd_i-d}{\phi_{i1}} + \inner{W\xd_i-H\sd_i-h}{\mu_{i1}+\gamma_{i}} \leq r_i + \tau\;\text{and}\; \left\|\begin{pmatrix}C\tr\phi_{i1} - H\tr(\mu_{i1} + \gamma_{i}) \\ W\tr(\mu_{i1} + \gamma_i)\end{pmatrix} \right\|_* \le \lambda.
	\end{array}
	\end{align*}
	A similar reasoning shows that the $i$-th member of the second constraint group in~\eqref{semiinf-linear} holds if and only if there exist $\phi_{i2} \in \coneXi^*$ and  $\mu_{i2}\in \coneX^*$  such that
	\begin{align*}
	&\begin{array}{clll} &  \inner{C\sd_i-d}{\phi_{i2}} + \inner{W\xd_i-H\sd_i-h}{\mu_{i2}} \leq r_i \quad \text{and} \quad \left\|\begin{pmatrix}C\tr\phi_{i2} - H\tr\mu_{i2} \\ W\tr\mu_{i2}\end{pmatrix} \right\|_* \le \lambda.
	\end{array}
	\end{align*}
	In summary, the worst-case expectation in the second line of~\eqref{wc_cvar} thus coincides with the optimal value of the finite conic program
	\begin{align*}
	\begin{array}{clll} \Inf{} &  {\eps} \lambda  + {1 \over N}\sum\limits_{i = 1}^{N} r_i \vspace{1mm}\\
	\st & 	\lambda \in\mathbb R_+,\;r_i\in\R,\; \phi_{i1},\phi_{i2}\in \coneXi^*, \;\mu_{i1},\mu_{i2},\gamma_{i}\in \coneX^* &\forall i \le N \\
	& \inner{C\sd_i-d}{\phi_{i1}} + \inner{W\xd_i-H\sd_i-h}{\mu_{i1}+\gamma_i} \leq r_i + \tau&\forall i \le N \vspace{1mm}\\	
	& \inner{C\sd_i-d}{\phi_{i2}} + \inner{W\xd_i-H\sd_i-h}{\mu_{i2}} \leq r_i,\;\; \theta=W\tr\gamma_i&\forall i\le N\vspace{1mm}\\
	& \left\|\begin{pmatrix}C\tr\phi_{i1} - H\tr(\mu_{i1} + \gamma_{i}) \\ W\tr(\mu_{i1} + \gamma_i)\end{pmatrix} \right\|_* \le \lambda,\;\; \left\|\begin{pmatrix}C\tr\phi_{i2} - H\tr\mu_{i2} \\ W\tr\mu_{i2}\end{pmatrix} \right\|_* \le \lambda&\forall i \le N \vspace{1mm}. 
	\end{array}
	\end{align*}
	The claim then follows by substituting this conic program into~\eqref{wc_cvar}. 
\end{proof}

{If $(\sd_i, \xd_i)\in\Xi$ for all $i\le N$, then the conic program~\eqref{cvar-lin} simplifies. In this case, the maximization problems in the last constraint group of~\eqref{semiinf-linear} all evaluate to zero, which implies that $\inner{C\sd_i-d}{\phi_{i2}} + \inner{W\xd_i-H\sd_i-h}{\mu_{i2}} \leq r_i$ reduces to $r_i\ge 0$, while the decision variables $\phi_{i2}$ and $\mu_{i2}$ as well as the constraints 
\[
	\left\|\begin{pmatrix}C\tr\phi_{i2} - H\tr\mu_{i2} \\ W\tr\mu_{i2}\end{pmatrix} \right\|_* \le \lambda
\]
can be omitted from \eqref{cvar-lin} for all $i\le N$.}

For stress test experiments it is often desirable to know the extremal distribution that achieves the worst-case risk in~\eqref{DRO-inv-opt}. The following theorem shows that this extremal distribution can be constructed systematically for any fixed $\theta\in\Theta$ by solving a finite convex optimization problem akin to~\eqref{cvar-lin}.

\begin{Thm}[Worst-case distribution for linear {hypotheses}]
	\label{thm:wc-dist}
	Under the assumptions of Theorem~\ref{thm:lin}, the worst-case risk in~\eqref{cvar-lin} corresponding to a fixed~$\theta\in\Theta$ coincides with the optimal value of a finite convex program, {\em i.e.},
	\begin{align}
	\label{wc-dist-lin}
	\begin{array}{lllll} \displaystyle \sup_{\Q\in \ball{p}{\Pem}{\eps}} \cvar_\alpha^\Q (\ell_\theta) ~= &\text{\rm max} & \displaystyle \frac{1}{\alpha N}\sum_{i=1}^N  ~\pi_{i1}\ell_\theta\left( {p_{i1} \over\pi_{i1}}, {q_{i1} \over\pi_{i1}}\right) \vspace{1mm}\\
	& \text{\rm s.t.} &  \pi_{ij} \in\R_+,\;\; p_{ij} \in\R^m,\;\; q_{ij} \in\R^n  &\forall i \le N,\; j\le 2\\[1mm]
	&&{p_{ij} \over\pi_{ij}}\in\mathbb S,\quad {q_{ij}\over\pi_{ij}}\in\X({p_{ij} \over\pi_{ij}}) &\forall i \le N,\; j\le 2\\[1mm]
	&& \displaystyle\pi_{i1}+\pi_{i2}=1, \quad \frac{1}{N}\sum_{i=1}^N\pi_{i1}=\alpha&\forall i\le N\\[1mm]
	&& \displaystyle \frac{1}{N}\sum_{i=1}^N \sum_{j=1}^2~\pi_{ij} \left\|\left(\begin{array}{c}\frac{p_{ij}}{\pi_{ij}}-\sd_i\\\frac{q_{ij}}{\pi_{ij}}-\xd_i\end{array}\right)\right\|
	\leq \eps.
	\end{array}
	\end{align}
	For any optimal solution $\{\pi\opt_{ij}, p\opt_{ij}, q\opt_{ij}\}$ of this convex program, the discrete distribution
	\begin{align*}
	\Q^\star\Let\frac{1}{N}\sum_{i=1}^{N} \sum_{j=1}^2 ~ \pi\opt_{ij}\delta_{\xi\opt_{ij}} \quad \text{ with }\quad \xi\opt_{ij}\Let\left({p\opt_{ij} \over\pi\opt_{ij}},{q\opt_{ij} \over\pi\opt_{ij}}\right)\tr
	\end{align*}
	belongs to the Wassertein ball $\ball{1}{\Pem}{\eps}$ and attains the supremum on the left hand side of~\eqref{wc-dist-lin}.
\end{Thm}
\begin{proof}
	As the loss function~\eqref{loss-lin} is proper and jointly concave in $x$ and $s$, we can use a similar reasoning as in~\cite[Theorem~4.5]{ref:MohKun-14} to show that the convex program on the right hand side of~\eqref{wc-dist-lin} coincides with the strong dual of~\eqref{cvar-lin} for any fixed~$\theta\in\Theta$. 
	If this convex program is solvable and $\{\pi\opt_{ij}, p\opt_{ij}, q\opt_{ij}\}$ is a maximizer, then $\Q\opt\in \ball{1}{\Pem}{\eps}$ due to~\cite[Corollary~4.7]{ref:MohKun-14}. It remains to be shown that $\cvar_\alpha^{\Q^\star}(\ell_\theta)$ is no smaller than~\eqref{cvar-lin}. Indeed, by the definition of $\cvar$ we have
	\begin{align*}
	\cvar_\alpha^{\Q^\star}(\ell_\theta)&=\inf_{\tau\in\R}\;\tau +     \frac{1}{\alpha N}\sum_{i=1}^N\sum_{j=1}^2\pi\opt_{ij}\max\left\{\ell_\theta\left( {p\opt_{ij} \over\pi\opt_{ij}}, {q\opt_{ij} \over\pi\opt_{ij}}\right)-\tau,0\right\}\\
	&=\sup_{0\le\nu_{ij}\le\pi\opt_{ij}}\left\{\frac{1}{\alpha N}\sum_{i=1}^N\sum_{j=1}^2\nu_{ij}\ell_\theta\left( {p\opt_{ij} \over\pi\opt_{ij}}, {q\opt_{ij} \over\pi\opt_{ij}}\right)~:~\alpha =\frac{1}{N}\sum_{i=1}^N\sum_{j=1}^2\nu_{ij}\right\}\\
	&\geq \frac{1}{\alpha N}\sum_{i=1}^N \pi\opt_{i1}\ell_\theta\left( {p\opt_{i1} \over\pi\opt_{i1}}, {q\opt_{i1} \over\pi\opt_{i1}}\right)=    \sup_{\Q\in \ball{p}{\Pem}{\eps}} \cvar_\alpha^\Q (\ell_\theta).
	\end{align*}
	Here, the second equality follows from strong linear programming duality, while the inequality follows from the feasibility of the solution $\nu_{i1}=\pi\opt_{i1}$ and $\nu_{i2}=0$ for $i\le N$.
\end{proof} 

{We close this section by generalizing Theorem~\ref{thm:lin} to the bounded rationality loss~\eqref{bounded-ratioinality-loss}.  The proof largely parallels that of Theorem~\ref{thm:lin} and is therefore omitted for brevity.}

\begin{Cor}[Linear {hypotheses} and bounded rationality loss]
	\label{cor:lin:bdd-r}
	Assume that $\F$ represents the class of linear {hypotheses with search space of the form~\eqref{eq:search-spaces-linear} and that Assumptions~\ref{As:conic} holds.} If the observer uses the bounded rationality loss~\eqref{bounded-ratioinality-loss} and measures risk using the $\cvar$ at level $\alpha\in(0,1]$, then the inverse optimization problem~\eqref{DRO-inv-opt} over the $1$-Wasserstein ball is equivalent to a variant of the conic program~\eqref{cvar-lin} with an additional decision variable $\tau\in\R_+$ and where the first constraint group is replaced with
	\[
	\inner{C\sd_i-d}{\phi_{i1}} + \inner{W\xd_i-H\sd_i-h}{\mu_{i1}+\gamma_i} \leq r_i + \tau+\delta\quad \forall i \le N.
	\]
\end{Cor}

\section{Quadratic {Hypotheses}}
\label{sec:quad}
Optimization problems with quadratic objectives abound in control~\cite{ref:AndMoo07}, statistics~\cite{friedman2001elements}, finance~\cite{markowitz1968portfolio} and many other application domains. Algorithms for inverse optimization that can learn quadratic objective functions from signal-response pairs are therefore of great practical interest. This motivates us to consider the class $\F$ of quadratic {hypotheses} of the form $F_\theta(s,x)\Let  \inner{x}{Q_{xx}x} + \inner{x}{Q_{x s}s}+\inner{q}{x}$, which are encoded by a parameter $\theta\Let(Q_{xx},Q_{x s}, q)$. {The corresponding search space should account for prior information on the agent's objective and should exclude $\theta=0$. Examples of tractable search spaces are listed below.

\begin{Ex}[Tractable search spaces]
	\label{Ex:Theta-quad}
	If $F$ is only known to be strongly convex in $x$, it is natural to set
	\begin{subequations}
	\label{eq:search-spaces-quadratic}
        \begin{align}
        \label{Theta-quad}
        		\Theta \Let \bigg\{ \theta = (Q_{xx},Q_{xs},q)\in \R^{n\times n}\times\R^{n\times m}\times \R^n : Q_{xx} \gesdp \mathbb I \bigg\}.
        \end{align}	
        Note that the normalization $Q_{xx} \gesdp \mathbb I$ is non-restrictive because a positive scaling of the objective function does not alter the agent's preferences. 
        
        If $F$ is only known to be bilinear in $s$ and $x$, it is natural to set
        \begin{align}
        \label{Theta-bilinear}
        		\Theta \Let \bigg\{ \theta = (Q_{xx},Q_{xs},q)\in \R^{n\times n}\times\R^{n\times m}\times \R^n : Q_{xx} \gesdp 0 ,~Q_{xs} = \mathbb I\bigg\},
        \end{align}	
        where the normalization $Q_{xs} = \mathbb I$ can always be enforced by redefining $s$ if necessary.
        
        If $F$ is close to a nominal objective function $\inner{x}{Q^0_{xx}x} + \inner{x}{Q^0_{x s}s}+\inner{q^0}{x}$, then we may set  
        \begin{align}
        \label{Theta-nominal}
        		\Theta \Let \bigg\{ \theta = (Q_{xx},Q_{xs},q)\in \R^{n\times n}\times\R^{n\times m}\times \R^n : Q_{xx} \gesdp 0 ,~ \| \theta-\theta_0\|\leq \Gamma\bigg\},
        \end{align}	
	where $\|\cdot\|$ denotes a generic norm, and $\Gamma$ captures the uncertainty of the nominal model~$\theta_0=(Q^0_{xx},Q^0_{xs},q^0)$. 
        \end{subequations}
\end{Ex}
}

When focusing on quadratic candidate objective functions, the {suboptimality loss}~\eqref{loss} reduces to
\begin{align}
\label{loss-quad}
\begin{array}{ll}
\ell_\theta(s,x) &\displaystyle = \inner{x}{Q_{xx}x + Q_{x s}s+q} - \min_{y \in \X(s)} \inner{y}{Q_{xx}y + Q_{x s}s+q} \\ & \displaystyle = \max_{y \in \X(s)} \inner{x}{Q_{xx}x + Q_{x s}s+q} - \inner{y}{Q_{xx}y + Q_{x s}s+q}.
\end{array}
\end{align}
As in Section~\ref{sec:lin}, we suppose that Assumption~\ref{As:conic} holds. In this setting the agent's decision problem~\eqref{opt_basic} constitutes a conic program and is therefore tractable for common choices of the cones $\coneXi$ and $\coneX$. In contrast, the inverse optimization problem~\eqref{DRO-inv-opt} is hard. In fact, it is already hard to evaluate the objective function of~\eqref{DRO-inv-opt} for a fixed $\theta$. As we work with quadratic objectives, throughout this section we use the $2$-norm on the signal-response space and the $2$-Wasserstein metric to measure distances of distributions.

\begin{Thm}[Intractability of~\eqref{DRO-inv-opt} for quadratic {hypotheses}]
	\label{thm:Quad_NP_hard_1}
	Assume that $\F$ represents the class of quadratic hypotheses with search space \eqref{Theta-quad} and that Assumption~\ref{As:conic} holds. If the observer uses the suboptimality loss~\eqref{loss} and measures risk using the $\cvar$ at level $\alpha\in(0,1]$, then evaluating the objective function of~\eqref{DRO-inv-opt} for a fixed $\theta\in\Theta$ is NP-hard even if $N=1$ (there is only one data point), $\alpha=1$ (the observer is risk-neutral), $\mathbb S$ is a singleton, $\X(s)$ is a polytope independent of $s$, and $Q_{xs}=0$.
	\end{Thm}

\begin{proof}
		The proof relies on a reduction from the NP-hard quadratic maximization problem~\cite{ref:MurKa-87}.  
	\\[-2mm]
	
	\begin{center}
	\fbox{\parbox{14cm}{ {\centering \textsc{Quadratic Maximization}\\}
			
			\textbf{Instance.} A positive definite matrix $ Q=Q\tr\succeq \mathbb I$. \\
			\textbf{Goal.} Evaluate $\max_{\|{x}\|_\infty\leq 1}\inner{x}{Qx}$.
		}}
		\end{center}
		
	Given an input $Q\succeq \mathbb I$ to the quadratic maximization problem, we construct an instance of the inverse optimization problem~\eqref{DRO-inv-opt} with $N=1$, $\alpha=1$, and Wasserstein radius $\eps=\sqrt{n}$, where 
	\begin{align*}
	&\sd_1\Let 0, \quad \xd_1\Let 0 ,\quad Q_{xx}\Let Q,\quad Q_{xs}\Let 0, \quad \mathbb S \Let \{0\}, \quad \X(s)  \Let \{x \in \R^n : \|x\|_\infty\leq 1 \}.
	\end{align*}
	Under this parametrization, the objective function of~\eqref{DRO-inv-opt} reduces to 
	\begin{align*}
	\sup_{\Q \in\ball{2}{\Pem}{\eps} }\risk^\Q(\ell_\theta)~=&\sup_{\Q \in\ball{2}{\Pem}{\eps} }\EE^\Q \left[ \max_{y \in \X(s)} \inner{x}{Q_{xx}x + Q_{x s}s+q} - \inner{y}{Q_{xx}y + Q_{x s}s+q} \right]\\
	=&\sup_{\Q \in \ball{2}{\Pem}{\eps}}\EE^\Q \left[ \Max{y\in\X(s)} \inner{x}{Qx} -\inner{y}{Qy}\right]\\
	\leq &\sup_{s\in\mathbb S, x\in\X(s) }  \Max{y\in\X(s)} \inner{x}{Qx} -\inner{y}{Qy} =\sup_{\|x\|_\infty\leq 1}\inner{x}{Qx},
	\end{align*}
	where the inequality in the third line follows from the inclusion $\ball{2}{\Pem}{\eps}\subseteq \M^2(\Xi)$, while the last equality holds because the innermost maximum is attained at $y=0$. As  $(s,x)\in\Xi$ if and only if $s=0$ and $\|x\|_\infty\leq 1$, we conclude that $(s,x)\in\Xi$ implies ${\|x\|_2}\leq \sqrt{n}$ and $\| (s,x)\|_2^2\leq n=\eps^2$. Moreover, as the empirical distribution $\Pem$ coincides with the Dirac point measure at $0$, the Wasserstein ball $\ball{2}{\Pem}{\eps}$ thus contains all distributions supported on $\Xi$, implying that the inequality in the above expression is in fact an equality. Hence, evaluating the objective function of~\eqref{DRO-inv-opt} is tantamount to solving the NP-hard quadratic maximization problem. This observation completes the proof.
\end{proof}

\begin{Cor}[Intractability of~\eqref{DRO-inv-opt} for bilinear {hypotheses}]
	\label{thm:Quad_NP_hard_2}
	{If all assumptions of Theorem~\ref{thm:Quad_NP_hard_1} hold but $\F$ denotes the class of bilinear hypotheses with search space \eqref{Theta-bilinear}, then evaluating the objective of~\eqref{DRO-inv-opt} for a fixed $\theta\in\Theta$ is NP-hard even if $N=1$, $\alpha=1$, $\S$ is a singleton, $\X(s)$ is a polytope independent of $s$ and~$Q_{xx}=0$.}
\end{Cor}
\begin{proof}
	The proof is similar to that of Theorem~\ref{thm:Quad_NP_hard_1} and omitted for brevity.
\end{proof}

Corollary~\ref{thm:Quad_NP_hard_2} asserts that the inverse optimization problem~\eqref{DRO-inv-opt} is intractable even if we focus on linear candidate objectives that depend on the exogenous signal~$s$. This finding contrasts with the tractability Theorem~\ref{thm:lin} for candidate objectives independent of $s$. The intractability results portrayed in Theorem~\ref{thm:Quad_NP_hard_1} and Corollary~\ref{thm:Quad_NP_hard_2} motivate us to devise a safe conic approximation for the inverse optimization problem~\eqref{DRO-inv-opt} with quadratic candidate objective functions.

\begin{Thm}[Quadratic {hypotheses} and suboptimality loss]
	\label{thm:Quad}
	Assume that $\F$ represents the class of quadratic {hypotheses with a search space of the form~\eqref{eq:search-spaces-quadratic} and that Assumption~\ref{As:conic} holds.} If the observer uses the suboptimality loss~\eqref{loss} and measures risk using the $\cvar$ at level $\alpha\in(0,1]$, then the following conic program provides a safe approximation for the distributionally robust inverse optimization problem~\eqref{DRO-inv-opt} over the $2$-Wasserstein ball:
	\begin{align}
	\label{cvar-quad}
	\begin{array}{llll} \text{\em minimize} & \displaystyle \tau+ {1\over\alpha}\left({\eps^2} \lambda  + {1 \over N}\sum\limits_{i = 1}^{N} r_i\right) \vspace{1mm}\\
	\text{\em subject to} &\theta {= (Q_{xx},Q_{xs},q)} \in\Theta,\; \lambda \in\mathbb R_+,\;\tau,r_i, \rho_{i1},\rho_{i2}\in\R &\forall i\leq N\vspace{1mm}\\
	&\phi_{i1},\phi_{i2}\in \coneXi^*,\;\mu_{i1},\mu_{i2}, \gamma_i \in \coneX^*,\;  \chi_{i1},\chi_{i2}\in\R^m,\; \zeta_{i1},\eta_{i1},\zeta_{i2}\in\R^n &\forall i\leq N\vspace{1mm}\\
	&\chi_{i1}=\frac{1}{2}(-C\tr\phi_{i1}+H\tr(\mu_{i1}+\gamma_{i}) -2\lambda \sd_i)&\forall i\leq N\vspace{1mm}\\
	&\zeta_{i1}=\frac{1}{2}(-q-W\tr\mu_{i1}-2\lambda \xd_i),\quad \eta_{i1}=\frac{1}{2}(q-W\tr\gamma_{i})&\forall i\leq N\vspace{1mm}\\
	&\rho_{i1}=\tau+r_i+\lambda\left( \inner{\xd_i}{\xd_i}+ \inner{\sd_i}{\sd_i}\right)+\inner{d}{\phi_{i1}}+\inner{h}{\mu_{i1}+\gamma_i}&\forall i\leq N\vspace{1mm}\\
	&\chi_{i2}=\frac{1}{2}(-C\tr\phi_{i2}+H\tr\mu_{i2} -2\lambda \sd_i)&\forall i\leq N\vspace{1mm}\\
	&\zeta_{i2}=\frac{1}{2}(-W\tr\mu_{i2}-2\lambda \xd_i)&\forall i\leq N\vspace{1mm}\\
	&\rho_{i2}=r_i+\lambda\left( \inner{\xd_i}{\xd_i} +\inner{ \sd_i}{\sd_i}\right)+\inner{d}{\phi_{i2}}+\inner{h}{\mu_i}&\forall i\leq N\vspace{1mm}\\
	&\begin{bmatrix}
	\lambda \mathbb I & -\frac{1}{2}Q_{xs}\tr & \frac{1}{2}Q_{xs}\tr& \chi_{i1}\\
	- \frac{1}{2}Q_{xs} &  \lambda \mathbb I -Q_{xx} & 0&  \zeta_{i1}\\
	\phantom{-} \frac{1}{2}Q_{xs} & 0 & Q_{xx} & \eta_{i1}\\
	\chi_{i1}\tr&\zeta_{i1}\tr&\eta_{i1}\tr& \rho_{i1} \end{bmatrix}\succeq  0,\quad \begin{bmatrix}
	\lambda \mathbb I & 0 & \chi_{i2}\\
	0 &  \lambda \mathbb I &  \zeta_{i2}\\
	\chi_{i2}\tr&\zeta_{i2}\tr& \rho_{i2} \end{bmatrix}\succeq  0&\forall i\leq N.
	\end{array}
	\end{align}
\end{Thm}

{Note that Theorem~\ref{thm:Quad} remains valid if $(\sd_i, \xd_i)\notin\Xi$ for some $i\le N$.}

\begin{proof}[Proof of Theorem~\ref{thm:Quad}]
	As in the proof of Theorem~\ref{thm:lin} one can show that the objective function of the inverse optimization problem~\eqref{DRO-inv-opt} coincides with a variant of~\eqref{wc_cvar} that involves the $2$-Wasserstein ball. In the remainder, we derive a safe conic approximation for the subordinate worst-case expectation problem
	\begin{align}
	\label{wc-exp-quad}
	\sup_{\Q \in \ball{2}{\Pem}{\eps}}  \EE^\Q\big[\max\{\ell_\theta(s,x) - \tau, 0\} \big].
	\end{align} 
	Duality arguments borrowed from~\cite[Theorem~4.2]{ref:MohKun-14} imply that the above infinite-dimensional linear program admits a strong dual of the form
	\begin{align}
	\label{semiinf-quad}
	\begin{array}{clll} \Inf{{\lambda\geq 0,r_i}} & {\eps^2} \lambda  + {1 \over N}\sum\limits_{i = 1}^{N} r_i \vspace{1mm}\\
	\st &  \Sup{s\in\mathbb S,\;x,y\in\X(s)} \inner{x}{Q_{xx}x + Q_{x s}s+q} - \inner{y}{Q_{xx}y + Q_{x s}s+q}-\tau\\
	&\qquad\qquad\qquad\qquad-\lambda\|(s,x)-(\sd_i,\xd_i)\|_2^2\leq r_i&\forall i\leq N  \vspace{1mm}\\
	&  \Sup{s\in\mathbb S,\;x\in\X(s)} -\lambda\|(s,x)-(\sd_i,\xd_i)\|_2^2\leq r_i&\forall i\leq N. \vspace{1mm}
	\end{array}
	\end{align}
	By using the definitions of $\mathbb S$ and $\X(s)$ {put forth in Assumption~\ref{As:conic}}, the $i$-th member of the first constraint group in~\eqref{semiinf-quad} is satisfied if and only if the optimal value of the maximization problem
	\begin{align}
	\label{conic-quadratic}
	&\begin{array}{clll} \Sup{{s,x,y}} & \inner{x}{Q_{xx}x + Q_{x s}s+q} - \inner{y}{Q_{xx}y + Q_{x s}s+q}-\tau-\lambda\| (s,x)-(\sd_i,\xd_i)\|_2^2 \vspace{1mm}\\
	\st &   C s  \gec{\coneXi} d,\; Wx \gec{\coneX} Hs + h,\; Wy \gec{\coneX} Hs + h
	\end{array}
	\end{align}
	does not exceed $r_i$. The Lagrangian of the conic quadratic program~\eqref{conic-quadratic} is defined as
	\begin{align}
	\label{eq:lagrangian}
	L(s,x,y; \phi_{i1},\mu_{i1}, \gamma_{i})~ \Let~&  \inner{x}{Q_{xx}x + Q_{x s}s+q} - \inner{y}{Q_{xx}y + Q_{x s}s+q}-\tau -\lambda\|(s,x)-(\sd_i,\xd_i)\|_2^2  \vspace{1mm}\\
	& \qquad +  \inner{Cs-d}{\phi_{i1}} + \inner{Wx-Hs-h}{\mu_{i1}} + \inner{Wy-Hs-h}{\gamma_{i}},\nonumber
	\end{align}
	and therefore~\eqref{conic-quadratic} can be expressed as a max-min problem of the form
	\begin{align*}
	\Sup{{s,x,y}} ~\Inf{\phi_{i1} \in \coneXi^*} ~ \Inf{\mu_{i1}, \gamma_{i} \in \coneX^*}~L(s,x,y; \phi_{i1},\mu_{i1}, \gamma_{i})
	~\leq ~ \Inf{\phi_{i1} \in \coneXi^*} ~ \Inf{\mu_{i1}, \gamma_{i} \in \coneX^*}  ~ \Sup{{s,x,y}} ~ L(s,x,y; \phi_{i1},\mu_{i1}, \gamma_{i}),
	\end{align*}
	where the inequality follows from weak duality. {As $\Xi$ contains a Slater point by virtue of Assumption~\ref{As:conic}, strong duality holds (meaning that the inequality collapses to an equality) if the Lagrangian is concave in $(s,x,y)$; see also Proposition~\ref{prop:exact} below.} We conclude that the $i$-th member of the first constraint group in~\eqref{semiinf-quad} holds if there exist $\phi_{i1} \in \coneXi^*$ and $\mu_{i1}, \gamma_{i} \in \coneX^*$ with $\sup_{s,x,y} L(s,x,y; \phi_{i1},\mu_{i1}, \gamma_{i})\leq r_i$. As the Lagrangian constitutes a quadratic function, this statement is satisfied if and only if 
	there are $\phi_{i1} \in \coneXi^*$, $\mu_{i1}, \gamma_{i} \in \coneX^*$, $\chi_{i1}\in\R^m$, $\zeta_{i1},\eta_{i1}\in\R^n$ and $\rho_{i1}\in\R$~with
	\begin{align*}
	\begin{array}{lll}
	\chi_{i1}=\frac{1}{2}(-C\tr\phi_{i1}+H\tr(\mu_{i1}+\gamma_{i})-2\lambda \sd_i)\vspace{1mm}\\
	\zeta_{i1}=\frac{1}{2}(-q-W\tr\mu_{i1}-2\lambda \xd_i),\quad \eta_{i1}=\frac{1}{2}(q-W\tr\gamma_{i})\vspace{1mm}\\
	\rho_{i1}=\tau+r_i+\lambda\left( \inner{\xd_i}{\xd_i}+\inner{\sd_i}{\sd_i}\right)+\inner{d}{\phi_{i1}}+\inner{h}{\mu_i+\gamma_i}\vspace{1mm}\\
	\begin{bmatrix}
	\lambda \mathbb I & -\frac{1}{2}Q_{xs}\tr & \frac{1}{2}Q_{xs}\tr& \chi_{i1}\\
	-\frac{1}{2}Q_{xs} &  \lambda \mathbb I-Q_{xx} & 0&  \zeta_{i1}\\
	\phantom{-} \frac{1}{2}Q_{xs} & 0 & Q_{xx} & \eta_{i1}\\
	\chi_{i1}\tr&\zeta_{i1}\tr&\eta_{i1}\tr& \rho_{i1} \end{bmatrix}\succeq  0.
	\end{array}
	\end{align*}
	Similarly, it can be shown that the $i$-th member of the second constraint group in~\eqref{semiinf-quad} is satisfied if and only if  there exist $\phi_{i2} \in \coneXi^*$, $\mu_{i2} \in \coneX^*$, $\chi_{i2}\in\R^m$, $\zeta_{i2}\in\R^n$ and $\rho_{i2}\in\R$ such that
	\begin{align}
	\label{eq:quad-to-be-dropped-if-consistent}
	\begin{array}{lll}
	\chi_{i2}=\frac{1}{2}(-C\tr\phi_{i2}+H\tr\mu_{i2}-2\lambda P_{s}\sd_i)\vspace{1mm}\\
	\zeta_{i2}=\frac{1}{2}(-W\tr\mu_{i2}-2\lambda \xd_i)\\
	\rho_{i2}=r_i+\lambda\left( \inner{\xd_i}{\xd_i}+\inner{\sd_i}{\sd_i}\right)+\inner{d}{\phi_{i2}}+\inner{h}{\mu_i}\vspace{1mm}\\
	\begin{bmatrix}
	\lambda \mathbb I & 0 & \chi_{i2}\\
	0 &  \lambda \mathbb I &  \zeta_{i2}\\
	\chi_{i2}\tr&\zeta_{i2}\tr& \rho_{i2} \end{bmatrix}\succeq  0.
	\end{array}
	\end{align}
	Replacing the semi-infinite constraints in~\eqref{semiinf-quad} with the respective semidefinite approximations shows that the worst-case expectation~\eqref{wc-exp-quad} is bounded above by the optimal value of the conic program
	\begin{align}
	\label{wc-exp-quad-sdp}
	\begin{array}{clll} \Inf{} &  {\eps^2} \lambda  + {1 \over N}\sum\limits_{i = 1}^{N} r_i \vspace{1mm}\\
	\st    &    \lambda \in\mathbb R_+,\;r_i, \rho_{i1},\rho_{i2}\in\R,\; \phi_{i1},\phi_{i2}\in \coneXi^*,\;\mu_{i1},\mu_{i2}, \gamma_{i} \in \coneX^* &\forall i\leq N\vspace{1mm}\\
	& \chi_{i1},\chi_{i2}\in\R^m,\; \zeta_{i1},\eta_{i1},\zeta_{i2}\in\R^n &\forall i\leq N\vspace{1mm}\\
	&\chi_{i1}=\frac{1}{2}(-C\tr\phi_{i1}+H\tr(\mu_{i1}+\gamma_{i}) -2\lambda \sd_i)&\forall i\leq N\vspace{1mm}\\
	&\zeta_{i1}=\frac{1}{2}(-q-W\tr\mu_{i1}-2\lambda \xd_i),\quad \eta_{i1}=\frac{1}{2}(q-W\tr\gamma_{i})&\forall i\leq N\vspace{1mm}\\
	&\rho_{i1}=\tau+r_i+\lambda\left( \inner{\xd_i}{\xd_i}+ \inner{\sd_i}{\sd_i}\right)+\inner{d}{\phi_{i1}}+\inner{h}{\mu_{i1}+\gamma_i}&\forall i\leq N\vspace{1mm}\\
	&\chi_{i2}=\frac{1}{2}(-C\tr\phi_{i2}+H\tr\mu_{i2} -2\lambda \sd_i)&\forall i\leq N\vspace{1mm}\\
	&\zeta_{i2}=\frac{1}{2}(-W\tr\mu_{i2}-2\lambda \xd_i) &\forall i\leq N\\
	&\rho_{i2}=r_i+\lambda\left( \inner{\xd_i}{\xd_i} +\inner{ \sd_i}{\sd_i}\right)+\inner{d}{\phi_{i2}}+\inner{h}{\mu_i}&\forall i\leq N\vspace{1mm}\\
	&\begin{bmatrix}
	\lambda \mathbb I & -\frac{1}{2}Q_{xs}\tr & \frac{1}{2}Q_{xs}\tr& \chi_{i1}\\
	- \frac{1}{2}Q_{xs} &  \lambda \mathbb I -Q_{xx} & 0&  \zeta_{i1}\\
	\phantom{-} \frac{1}{2}Q_{xs} & 0 & Q_{xx} & \eta_{i1}\\
	\chi_{i1}\tr&\zeta_{i1}\tr&\eta_{i1}\tr& \rho_{i1} \end{bmatrix}\succeq  0,\quad \begin{bmatrix}
	\lambda \mathbb I & 0 & \chi_{i2}\\
	0 &  \lambda \mathbb I &  \zeta_{i2}\\
	\chi_{i2}\tr&\zeta_{i2}\tr& \rho_{i2} \end{bmatrix}\succeq  0&\forall i\leq N.
	\end{array}
	\end{align}
	The claim then follows by substituting~\eqref{wc-exp-quad-sdp} into a suitable worst-case $\cvar$ formula akin to~\eqref{wc_cvar}.
\end{proof}

{If $(\sd_i, \xd_i)\in\Xi$ for all $i\le N$, then the conic program~\eqref{cvar-quad} simplifies. In this case, the maximization problems in the last constraint group of~\eqref{semiinf-quad} all evaluate to zero, which implies that the constraints~\eqref{eq:quad-to-be-dropped-if-consistent} reduce to $r_i\ge 0$, while the decision variables $\phi_{i2}$, $\mu_{i2}$, $\chi_{i2}$, $\zeta_{i2}$ and $\rho_{i2}$ can be omitted from \eqref{cvar-quad} for all $i\le N$.

In spite of the hardness results outlined in Theorems~\ref{thm:Quad_NP_hard_1} and~\ref{thm:Quad_NP_hard_2}, the following proposition shows that the tractable approximation of Theorem~\ref{thm:Quad} is sometimes exact.

\begin{Prop}[Ex post exactness guarantee]
\label{prop:exact}
	If an optimal solution to the safe conic approximation~\eqref{cvar-quad} from Theorem~\ref{thm:Quad} satisfies the strict inequality  
	\begin{align}
	\label{exact-ineq}
	\begin{bmatrix}
	\lambda \mathbb I & -\frac{1}{2}{Q_{xs}}\tr & \frac{1}{2}{Q_{xs}}\tr\\
	-\frac{1}{2}Q_{xs} &  \lambda \mathbb I-Q_{xx} & 0\\
	\phantom{-}\frac{1}{2}Q_{xs} & 0 & Q_{xx}\end{bmatrix} \succ 0,
	\end{align}
	then~\eqref{cvar-quad} is equivalent to the original distributionally robust inverse optimization problem~\eqref{cvar-quad}.
\end{Prop}

Our computational experiments suggest that the ex post exactness condition~\eqref{exact-ineq} is often satisfied if the Wasserstein radius $\eps$ is not too large.

\begin{proof}[Proof of Proposition~\ref{prop:exact}]
From the proof of Theorem~\ref{thm:Quad} we conclude that the optimal value of the distributionally robust inverse optimization problem~\eqref{DRO-inv-opt} can be represented as~$\inf_{\theta, \lambda, \tau} \varphi(\theta,\lambda,\tau)$, where
\begin{align*}
	\varphi(\theta,\lambda,\tau)  \Let &\lambda\eps^2 + {1\over N} \sum_{i=1}^{N} \Sup{s\in\mathbb S,\, x,y \in\X(s)} 
	\max\Big\{ \inner{x}{Q_{xx}x + Q_{x s}s+q}\- \inner{y}{Q_{xx}y + Q_{x s}s+q} -\tau,0\Big\}\\ 
	&\hspace{5cm} -\lambda\|(s,x)-(\sd_i,\xd_i)\|^2_2
\end{align*}
for $\theta\in\Theta$ and $\lambda\geq 0$, and $\varphi(\theta,\lambda,\tau)=\infty$ otherwise. By construction, $\varphi$ is jointly convex in $\theta=(Q_{xx},Q_{xs}, q)$, $\lambda$ and $\tau$. Moreover, it is clear that the optimal value of the finite convex program~\eqref{cvar-quad} can be represented as~$\inf_{\theta, \lambda, \tau} \wh \varphi(\theta,\lambda,\tau)$, where $\wh \varphi(\theta,\lambda,\tau)$ denotes the infimum of~\eqref{cvar-quad} when the decision variables $\theta$, $\lambda$ and $\tau$ are fixed. The proof of Theorem~\ref{thm:Quad} further implies that $\wh \varphi$ coincides with $\varphi$ whenever the Lagrangian~\eqref{eq:lagrangian} is concave in $(s,x,y)$, which is the case if and only if
	\begin{align}
	\label{region}
	\begin{bmatrix}
	-\lambda \mathbb I & \frac{1}{2}Q_{xs}\tr & -\frac{1}{2}Q_{xs}\tr\\
	\phantom{-}\frac{1}{2}Q_{xs} &  Q_{xx} - \lambda \mathbb I & 0\\
	-\frac{1}{2}Q_{xs} & 0 & -Q_{xx}\end{bmatrix} \preceq 0 .
	\end{align} 
Note also that $\wh \varphi(\theta,\lambda,\tau)=\infty$ whenever \eqref{region} is violated because \eqref{region} is implied by the constraints of~\eqref{cvar-quad}. In summary, $\varphi$ and $\wh\varphi$ are both convex functions satisfying
\begin{align*}
	\wh \varphi(\theta,\lambda,\tau) = \left\{ \begin{array}{ll} \varphi(\theta,\lambda,\tau) & \text{if \eqref{region} holds,} \\
	+\infty & \text{otherwise.}\end{array} \right.
\end{align*}
Thus, any minimizer of $\wh \varphi$ satisfying the strict inequality~\eqref{exact-ineq} resides in the interior of the region where the convex functions $\wh\varphi$ and $\varphi$ coincide and must therefore also minimize $\varphi$. 
\end{proof}
}

{Finally, we generalize Theorem~\ref{thm:Quad} to the bounded rationality loss~\eqref{bounded-ratioinality-loss}. The proof largely parallels that of Theorem~\ref{thm:Quad} and is therefore omitted for brevity.}

\begin{Cor}[Quadratic {hypotheses} and bounded rationality loss]
	\label{cor:quad:bdd-r}
	Assume that $\F$ represents the class of quadratic {hypotheses with a search space of the form~\eqref{eq:search-spaces-quadratic} and that Assumption~\ref{As:conic} holds.} If the observer uses the bounded rationality loss~\eqref{bounded-ratioinality-loss} and measures risk using the $\cvar$ at level $\alpha\in(0,1]$, then the inverse optimization problem~\eqref{DRO-inv-opt} over the $2$-Wasserstein ball is is conservatively approximated by a variant of the conic program~\eqref{cvar-quad} where $\tau\in\R_+$ and the defining equation of $\rho_i$ is replaced with
	\[
	\rho_{i1}=\tau+r_i+\delta+\lambda\left( \inner{\xd_i}{\xd_i}+ \inner{\sd_i}{\sd_i}\right)+\inner{d}{\phi_{i1}}+\inner{h}{\mu_{i1}+\gamma_i} \quad \forall i\leq N.
	\]
\end{Cor}

{The distributionally robust inverse optimization problem~\eqref{DRO-inv-opt} also admits a safe convex approximation when the search space consists of a class of convex hypotheses of the type considered in~\cite{keshavarz2011imputing}. Due to space restrictions we relegate this discussion to Appendix~\ref{app:convex}.}

\section{Numerical Experiments}
\label{sec:sim}

{
We now compare the proposed distributionally robust approach to inverse optimization against the state-of-the-art techniques described in~\cite{aswani2015inverse} and \cite{bertsimas2012inverse}. The first experiment aims to learn a linear hypothesis from imperfect training samples, where the imperfection is explained by measurement noise or the agent's bounded rationality. Similarly, the second experiment endeavors to learn a quadratic hypothesis from imperfect training samples, where the imperfection is explained by measurement noise or model uncertainty. 

All experiments are run on an Intel XEON CPU with 3.40GHz clock speed and 16GB of RAM. All linear, quadratic and second-order cone programs are solved with CPLEX~12.6, and all semidefinite programs are solved with MOSEK~8. 
In order to ensure that our experiments are reproducible, we make the underlying source codes available at \url{https://github.com/sorooshafiee/InverseOptimization}. 


\subsection{Learning a Linear Hypothesis}
\label{subsec:linear}
The goal of this experiment is to learn a linear hypothesis from imperfect training samples where the measured responses represent feasible perturbations of the true optimal responses. The perturbations can be explained by measurement noise or by the agent's bounded rationality. We argue that different causes of the perturbations necessitate different inverse optimization models.

\subsubsection{Consistent Noisy Measurements} 
\label{subsubsec:meas}
{\bf Decision problem of the agent:} We assume that the agent's true objective function is $F(s,x)=\inner{\theta\opt}{x}$ for some~$\theta\opt$ in the vicinity of a nominal model~$\theta_0$. The nominal model is sampled uniformly from $\Theta_0\Let \big\{\theta \in \R^n: \|\theta_0\|_\infty \in [1,5]\big\}$, while $\theta\opt$ is sampled uniformly form $\Theta \Let \big\{\theta \in \R^n : \|\theta-\theta_0\|_\infty \le 1\big\}$. This construction implies that $\theta\opt\neq 0$ almost surely. We also assume that the agent's feasible set is given by $\X(s)=\{x\in\R^n: \| x \|_\infty\leq 1,~ A x \geq s\}$, where the constraint matrix $A$ is sampled uniformly from the hypercube $[-1,1]^{m\times n}$. This feasible set can be brought to the standard form of Assumption~\ref{As:conic} by setting
\[
	W=(\mathbb I, -\mathbb I, A\tr)\tr \in\R^{(2n+m)\times n},~ H=(0,0,\mathbb I)\tr\in\R^{(2n+m)\times m},~ h=(-1,-1,0)\tr\in\R^{2n+m} \text{ and } \mathcal K=\R_+^{2n+m}.
\]
For a fixed signal $s$, we denote the optimal value of the agent's true decision problem by $z\opt(s)$.

{\bf Generation of training samples:} A single signal is constructed as $s = Av_1$, where the auxiliary random variable $v_1$ follows independent uniform distribution on $[-1,1]^n$. Thus, if $a_i$, $i\le m$, denote the rows of the constraint matrix $A$, then the support of $s$ can be expressed as $\S=\{s\in\R^m: |s_i|\leq \|a_i\|_1~\forall i\leq m\}$. This support can be brought to the standard form of Assumption~\ref{As:conic} by setting
\[
	C=(\mathbb I,-\mathbb I)\tr\in\R^{2m\times m},~d=(\|a_1\|_1,\ldots,\|a_m\|_1,\|a_1\|_1,\ldots,\|a_m\|_1)\tr\in\R^{2m} \text{ and } \mathcal C=\R_+^{2m}.
\]
As $v_1\in\X(s)$ by construction, problem~\eqref{opt_basic} is feasible for every $s\in\S$. We assume that the agent's response to $s$ is noisy in the sense that it constitutes a random $\delta$-suboptimal solution to~\eqref{opt_basic} for $\delta=1$. Specifically, we assume that $x$ is obtained as a solution of an auxiliary optimization problem
\[
	\min_{x\in\X(s)} \left\{ \inner{\theta_{\rm rand}}{x} :   \inner{\theta^\star}{x} \leq z\opt(s) + \delta \right\},
\]
which minimizes a linear cost over the set of all $\delta$-suboptimal solutions to~\eqref{opt_basic}. The gradient $\theta_{\rm rand}$ of the cost is sampled uniformly from $[-1,1]^n$. By construction, we have $x\in\X(s)$, which implies that $(s,x)\in\Xi$. Thus, the measurement noise is consistent with the know support of the exact signal-response pairs. The imperfect consistent training samples $(\sd_i, \xd_i)$, $i\le N$, are now generated independently using the above procedure.

{\bf Decision problem of the observer:} The observer aims to identify a linear hypothesis $F_\theta(s,x) \Let \inner{\theta}{x}$, $\theta\in\Theta$, that best predicts the agent's responses to new signals, where the search space $\Theta$ is set to the $\infty$-norm ball around the nominal model $\theta_0$ described above. We assume that the observer minimizes the suboptimality loss~\eqref{loss} and uses the expected value to measure risk. Moreover, the observer solves the distributionally robust inverse optimization problem~\eqref{DRO-inv-opt} over a $1$-Wasserstein ball around the empirical distribution on the training samples, where the $\infty$-norm is used as the transportation cost on $\Xi$. Note that this problem can be reformulated as the tractable linear program \eqref{cvar-lin} by virtue of Theorem~\ref{thm:lin}.

{\bf Out-of-sample risk:} To assess the quality of an estimator $\wh \theta_N$ obtained from~\eqref{cvar-lin}, we evaluate its out-of-sample risk $\EE^{\PP_{\rm out}} (\ell_{\wh \theta_N})$ both with respect to the predictability loss~\eqref{loss_predictability} and the suboptimality loss~\eqref{loss}, where $\P_{\rm out}$ represents the distribution of a single test sample $(s,x)$ independent of the training samples, and where $x$ is an exact (non-noisy) response to $s$ in~\eqref{opt_basic}. In practice, the out-of-sample risk cannot be evaluated analytically but only numerically by using 1{,}000 independent test samples from $\P_{\rm out}$. By relying on non-noisy test samples, we assess how well the estimator $\wh \theta_N$ can predict the agent's {\em exact} optimal responses.

\begin{figure*} [t!]
	\centering
	\hspace{-6pt} \subfigure[Impact of the Wasserstein radius on the suboptimality and predictibility risk]{\label{fig:meas:index} \includegraphics[width=0.32\columnwidth]{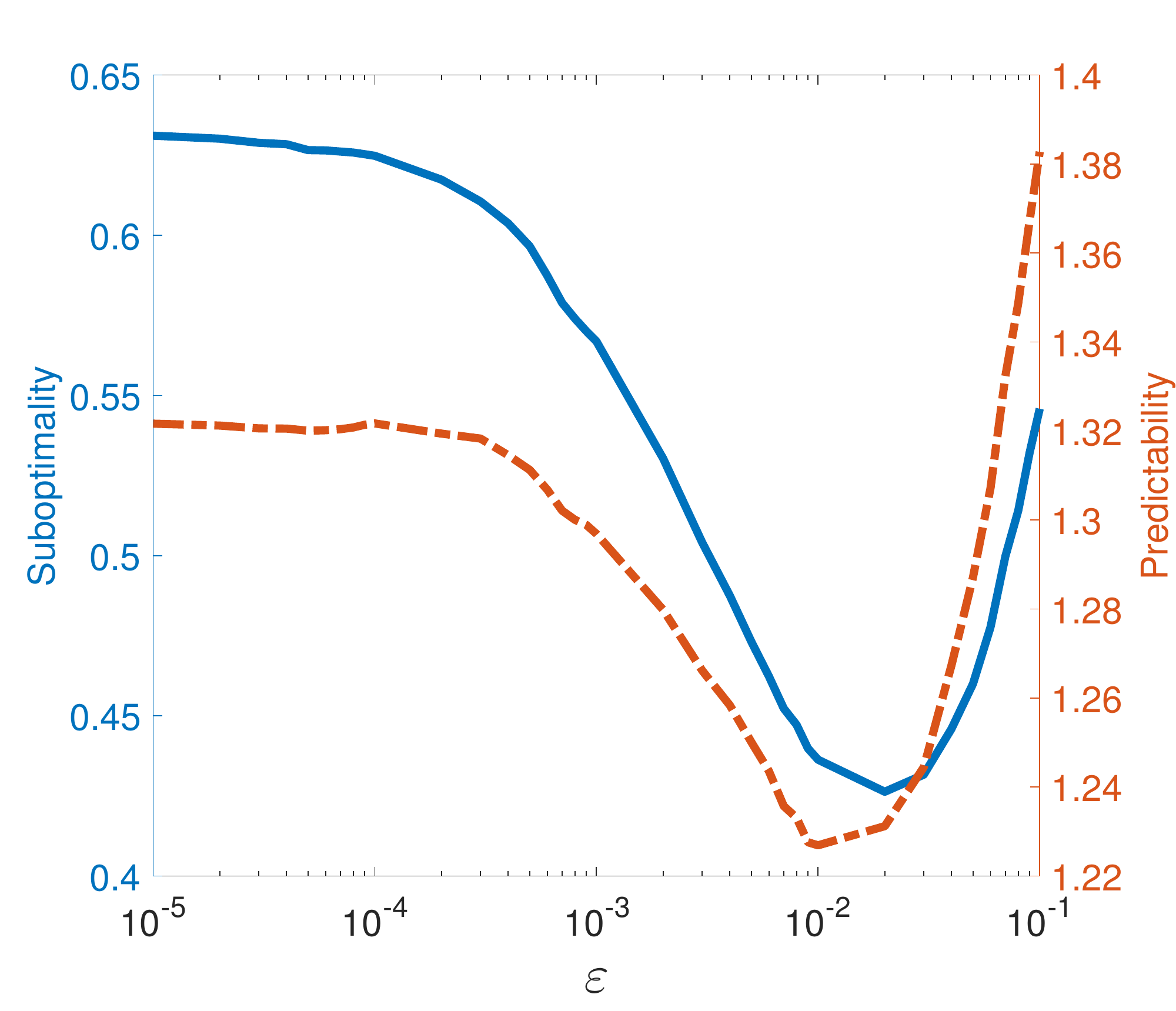}}
	\subfigure[Suboptimality learning curve]{\label{fig:meas:learn:sub} \includegraphics[width=0.32\columnwidth]{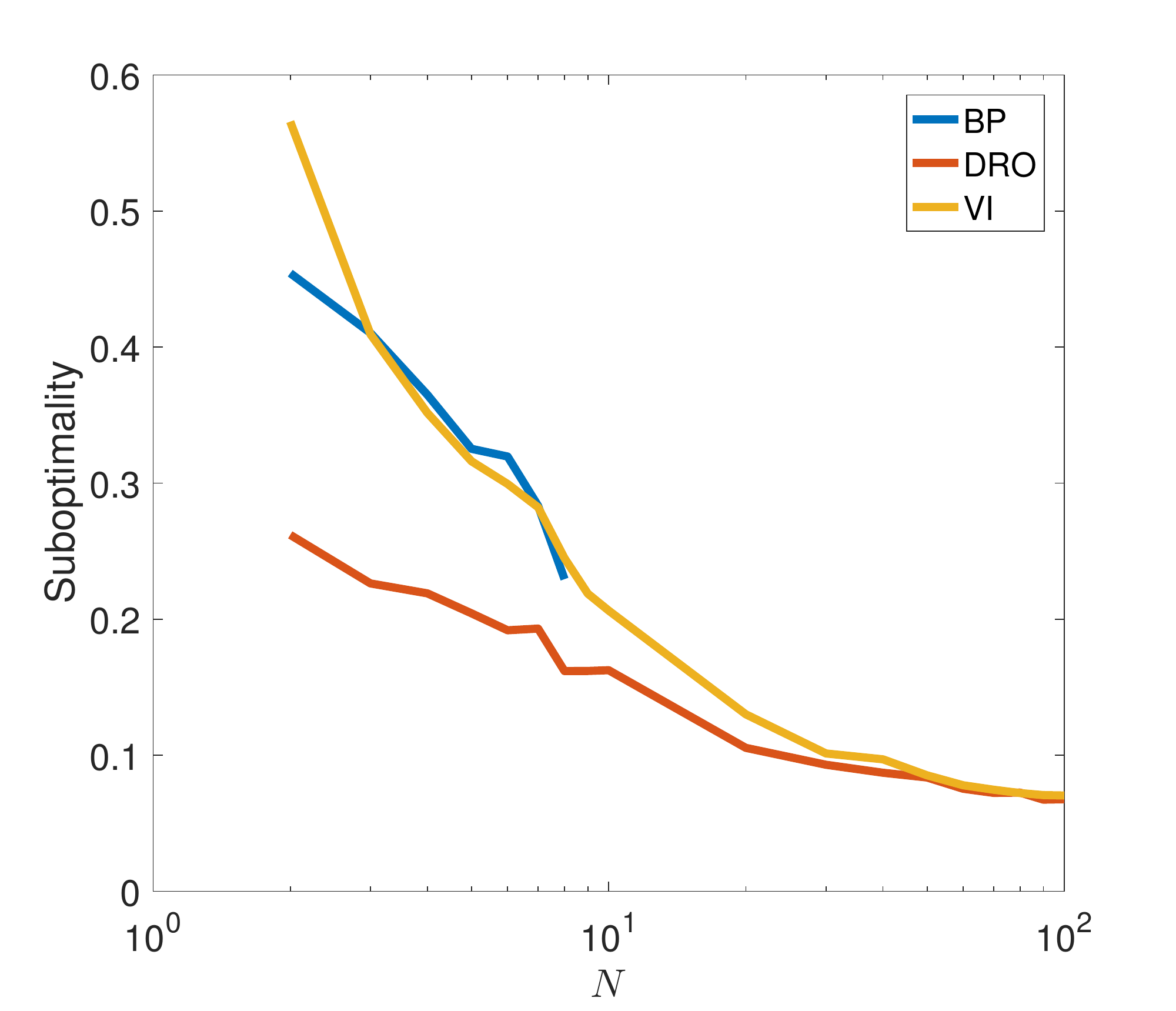}} 
	\subfigure[Predictibility learning curve]{\label{fig:meas:learn:pred} \includegraphics[width=0.32\columnwidth]{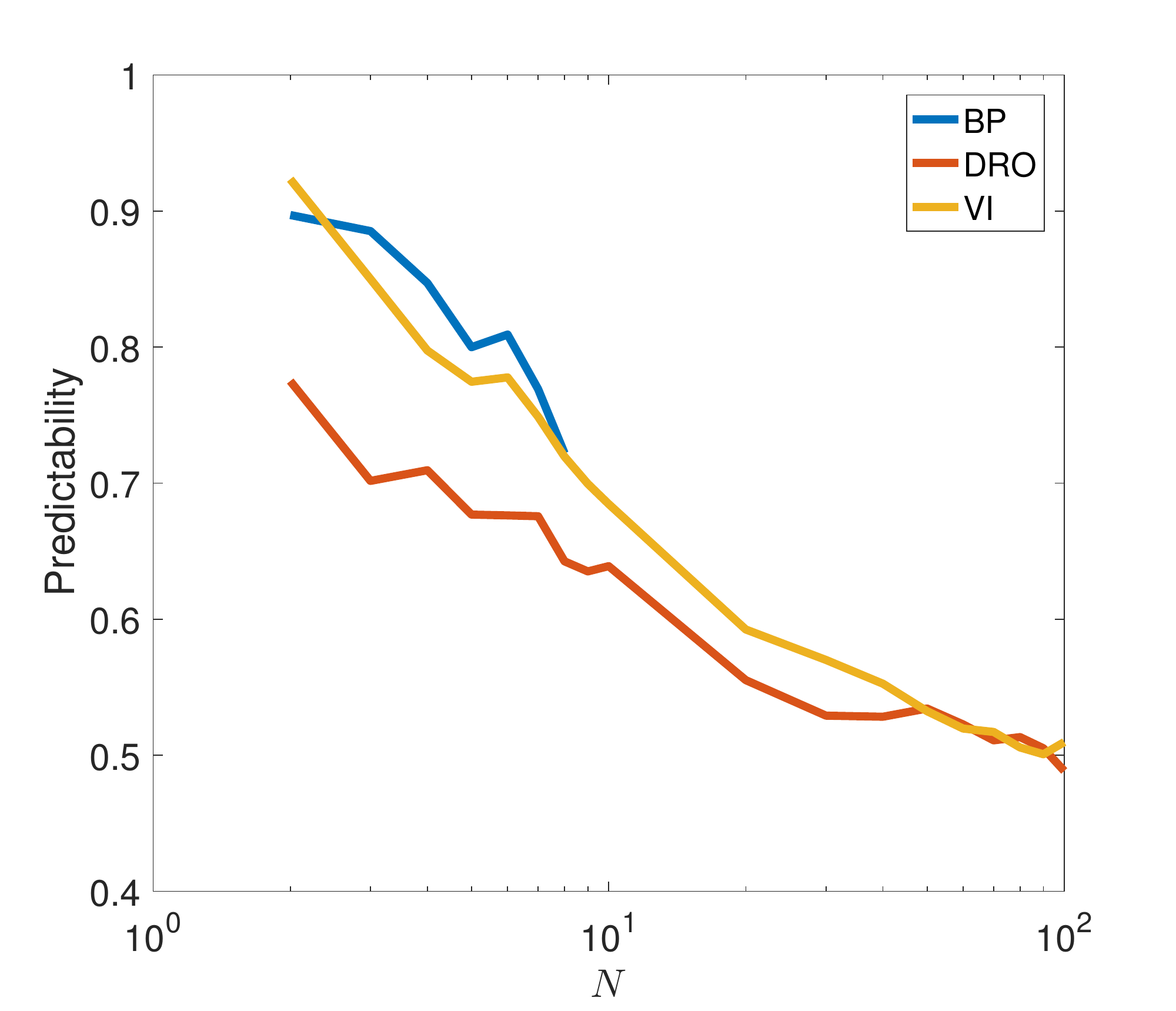}} 
	\caption{Out-of-sample suboptimality and predictability risks in the presence of imperfect consistent training samples and perfect test samples}
	\label{fig:meas}
\end{figure*}

{\bf Results:} All numerical results are averaged across $100$ independent problems instances $\{\theta_0,\theta\opt,A\}$. The first experiment involves $m=50$ signal variables, $n=50$ decision variables and $N=10$ training samples. Figure~\ref{fig:meas:index} shows how the out-of-sample risk of the optimal estimator $\wh \theta_N(\eps)$ obtained from~\eqref{cvar-lin} changes with the radius $\eps$ of the underlying Wasserstein ball. This experiment suggests that both the predictability and suboptimality risks can be significantly reduced by using a distributionally robust inverse optimization model with a judiciously calibrated ambiguity set. Unfortunately, Figure~\ref{fig:meas:index}, from which the optimal Wasserstein radii could be read off easily, is not available in the training phase as its construction requires large amounts of test samples. Instead, the Wasserstein radii offering the lowest out-of-sample risk must also be estimated from the training data. To this end, we use the following {\em $k$-fold cross validation} algorithm: 

Partition $\xid_1,\ldots,\xid_N$ into $k$ folds, and repeat the following procedure for each fold $i=1,\ldots,k$. Use the $i$-th fold as a validation dataset and merge the remaining $k-1$ folds to a training dataset of size $N_i$. Using only the $i$-th training dataset, solve~\eqref{cvar-lin} for a large but finite number of candidate radii $\eps\in\mathcal E$ to obtain an estimator $\wh \theta_{N_i}(\eps)$. Use the $i$-th validation dataset to estimate the out-of-sample risk of $\wh \theta_{N_i}(\eps)$ for each $\eps \in\mathcal E$. Set $\wh \eps_N^i$ to any $\eps\in\mathcal E$ that minimizes this quantity. After all folds have been processed, set $\wh \eps_N$ to the average of the $\wh \eps_N^i$ for $i=1,\ldots,k$, and re-solve~\eqref{cvar-lin} with $\eps=\wh \eps_N$ using all $N$ samples. Report the optimal solution $\wh\theta_N$ and the optimal value $\wh J_N$ of \eqref{cvar-lin} as the recommended estimator and its corresponding certificate, respectively. Throughout all experiments we set $k = \min \{ 5, N\}$ and $\mathcal E=\{\eps=b\cdot 10^c:b\in\{1,5\},\; c\in\{-4,-3,-2,-1\}\}$. 

For brevity, we henceforth refer to the above cross-validation scheme as the {\em distributionally robust optimization} (DRO) approach. In the following, we compare the resulting DRO estimator against two state-of-the-art estimators from the literature. The first estimator is obtained from the {\em variational inequality} (VI) approach proposed in~\cite{bertsimas2012inverse}, which minimizes the first-order loss~\eqref{loss_bertsimas} averaged across all training samples. By \cite[Theorem~3]{bertsimas2012inverse}, the resulting inverse optimization problem is equivalent to the tractable linear program
\begin{align*}
\begin{array}{lll}
\text{\rm minimize} & \displaystyle \frac{1}{N}\sum\limits_{i = 1}^{N} |r_i| & \\
\text{\rm subject to} & \inner{W \xd_i - H \sd_i - h}{\gamma_i} \leq r_i & \forall i \leq N \\
& W^\top \gamma_i = \theta & \forall i \leq N \\
& \gamma_i \ge 0 & \forall i \leq N \\ 
& \theta \in \Theta. & \\
\end{array}
\end{align*}
Note that as all training samples are consistent, that is, $(\sd_i,\xd_i)\in\Xi$ for $i=1,\ldots,n$, one can show that all residuals $r_i$ are automatically non-negative, and the above linear program can be viewed as a special instance of \eqref{DRO-inv-opt} that minimizes the first-order loss, where the risk measure is set to the expected value and the Wasserstein radius is set to zero. If there were inconsistent training samples that fall outside of $\Xi$, on the other hand, the absolute values of the residuals would be needed to prevent unboundedness.

The second estimator is obtained from the {\em bilevel programming} (BP) approach proposed in \cite{aswani2015inverse}, which minimizes the predictability loss~\eqref{loss_predictability} averaged across all training samples. As shown in~\cite[Section~2.2]{aswani2015inverse}, the resulting inverse optimization problem is equivalent to the optimistic bilevel program
\begin{align*}
\begin{array}{ll}
\text{\rm minimize} & \displaystyle \frac{1}{N} \sum\limits_{i=1}^N \|\xd_i - y_i\|_2^2 \\ 
\text{\rm subject to} & y_i \in \arg\min\limits_{z} \Big\{ \inner{z}{\theta} ~:~ Wz \ge H\sd_i + h \Big\} \\
& \theta \in \Theta.
\end{array}
\end{align*}
Note that this bilevel program can be viewed as a special instance of \eqref{DRO-inv-opt} that minimizes the predictability loss, where the risk measure is set to the expected value and the Wasserstein radius is set to zero.

The above bilevel program can be reformulated as a mixed integer quadratic program by replacing the `$\arg\min$'-constraint with the Karush-Kuhn-Tucker optimality conditions of the agent's linear program. Indeed, note that the resulting complementary slackness conditions can be linearized by introducing binary variables that identify the binding constraints. However, this approach requires big-$M$ constants to bound the optimal dual variables. As valid big-$M$ constants are difficult to obtain in general and as overly conservative big-$M$ constants lead to numerical instability, we use here the YALMIP interface for bilevel programming, which calls a dedicated branch and bound algorithm that branches directly on the complementarity slackness conditions~\cite{ref:YALMIP}. Throughout our experiments we limit all branch and bound calculations to $2{,}000$ iterations.

Figures~\ref{fig:meas:learn:sub} and \ref{fig:meas:learn:pred} visualize the suboptimality and predictability {\em learning curves}, respectively, which capture how the out-of-sample risk of the different estimators changes with the number of training samples. As optimistic bilevel programs are NP-hard even when all objective and constraint functions are linear~\cite{ref:Ben-Ayed-90}, this experiment focuses on problem instances with $n=m=10$. Even for these moderate problem dimensions, however, the inverse optimization problems associated with the BP approach fails to find a feasible solution for more than 10 training samples. In contrast, all other approaches lead to tractable linear programs. Figure~\ref{fig:meas:learn:sub} shows that the DRO approach dominates the VI and BP approaches uniformly across all samples sizes in terms of out-of-sample suboptimality risk. This is reassuring as the DRO approach actually minimizes suboptimality risk. However, Figure~\ref{fig:meas:learn:pred} suggests that the DRO approach wins even in terms of out-of-sample predictability risk. This is perhaps surprising because, unlike the BP approach, the DRO approach only minimizes an approximation of the predictability loss. We conclude that injecting distributional robustness may be more beneficial for out-of-sample performance than using the correct loss function.

Tables~\ref{table:subopt:imp} and \ref{table:pred:imp} report the out-of-sample suboptimality and predictability risks, respectively, for the DRO and VI estimators based on $N=10$ training samples and for different signal and response dimensions. The BP approach is excluded from this comparison due to its intractability. We observe that the DRO estimator always attains the lowest suboptimality risk and often the lowest predictability risk. Whenever the VI estimator wins, the predictability risks of the VI and DRO estimators are in fact almost identical.

\begin{table}
	\centering
	\caption{Out-of-sample suboptimality risk in the presence of noisy consistent measurements}
	\label{table:subopt:imp}
	\begin{tabular}{cclllll}\hline 
		& & \multicolumn{5}{c}{$m$} \\ \cline{3-7} 
		$n$ & Methods & 10 & 20 & 30 & 40 & 50  \\ \noalign{\vskip 1pt} \hline \noalign{\vskip 1pt} 
		\multirow{3}{*}{10} & VI & $1.4\times 10^{-1} $& $1.8\times 10^{-1} $& $1.4\times 10^{-1} $& $1.2\times 10^{-1} $& $1.3\times 10^{-1} $ \\ 
		& DRO & \cellcolor{gray!25}{$1.2\times 10^{-1}$}& \cellcolor{gray!25}{$1.3\times 10^{-1}$}& \cellcolor{gray!25}{$1.1\times 10^{-1}$}& \cellcolor{gray!25}{$8.9\times 10^{-2}$}& \cellcolor{gray!25}{$9.5\times 10^{-2}$} \\ \noalign{\vskip 1pt} \hline \noalign{\vskip 1pt} 
		\multirow{3}{*}{20} & VI & $3.0\times 10^{-1} $& $3.5\times 10^{-1} $& $3.2\times 10^{-1} $& $2.5\times 10^{-1} $& $2.4\times 10^{-1} $ \\ 
		& DRO & \cellcolor{gray!25}{$2.4\times 10^{-1}$}& \cellcolor{gray!25}{$2.9\times 10^{-1}$}& \cellcolor{gray!25}{$2.4\times 10^{-1}$}& \cellcolor{gray!25}{$1.9\times 10^{-1}$}& \cellcolor{gray!25}{$1.9\times 10^{-1}$} \\ \noalign{\vskip 1pt} \hline \noalign{\vskip 1pt} 
		\multirow{3}{*}{30} & VI & $3.8\times 10^{-1} $& $4.4\times 10^{-1} $& $4.6\times 10^{-1} $& $4.2\times 10^{-1} $& $3.6\times 10^{-1} $ \\ 
		& DRO & \cellcolor{gray!25}{$3.2\times 10^{-1}$}& \cellcolor{gray!25}{$3.7\times 10^{-1}$}& \cellcolor{gray!25}{$3.7\times 10^{-1}$}& \cellcolor{gray!25}{$3.3\times 10^{-1}$}& \cellcolor{gray!25}{$3.1\times 10^{-1}$} \\ \noalign{\vskip 1pt} \hline \noalign{\vskip 1pt} 
		\multirow{3}{*}{40} & VI & $4.4\times 10^{-1} $& $5.5\times 10^{-1} $& $6.0\times 10^{-1} $& $5.4\times 10^{-1} $& $5.1\times 10^{-1} $ \\ 
		& DRO & \cellcolor{gray!25}{$3.7\times 10^{-1}$}& \cellcolor{gray!25}{$4.5\times 10^{-1}$}& \cellcolor{gray!25}{$4.6\times 10^{-1}$}& \cellcolor{gray!25}{$4.4\times 10^{-1}$}& \cellcolor{gray!25}{$4.5\times 10^{-1}$} \\ \noalign{\vskip 1pt} \hline \noalign{\vskip 1pt} 
		\multirow{3}{*}{50} & VI & $5.0\times 10^{-1} $& $6.4\times 10^{-1} $& $6.5\times 10^{-1} $& $7.1\times 10^{-1} $& $6.2\times 10^{-1} $ \\
		& DRO & \cellcolor{gray!25}{$4.1\times 10^{-1}$}& \cellcolor{gray!25}{$5.1\times 10^{-1}$}& \cellcolor{gray!25}{$5.4\times 10^{-1}$}& \cellcolor{gray!25}{$5.6\times 10^{-1}$}& \cellcolor{gray!25}{$5.4\times 10^{-1}$} \\ \noalign{\vskip 1pt} \hline \noalign{\vskip 1pt} 
	\end{tabular} 
\end{table}

\begin{table}
	\centering 
	\caption{Out-of-sample predictability risk in the presence of consistent noisy measurements}
	\label{table:pred:imp}
	\begin{tabular}{cclllll}\noalign{\vskip 1pt} \hline \noalign{\vskip 1pt} 
		& & \multicolumn{5}{c}{$m$} \\ \cline{3-7} 
		$n$ & Methods & 10 & 20 & 30 & 40 & 50  \\ \noalign{\vskip 1pt} \hline \noalign{\vskip 1pt} 
		\multirow{3}{*}{10} & VI & $6.1\times 10^{-1} $& $4.6\times 10^{-1} $& $3.4\times 10^{-1} $& $2.6\times 10^{-1} $& $2.4\times 10^{-1} $ \\ 
		& DRO & \cellcolor{gray!25}{$5.6\times 10^{-1}$}& \cellcolor{gray!25}{$4.2\times 10^{-1}$}& \cellcolor{gray!25}{$3.2\times 10^{-1}$}& \cellcolor{gray!25}{$2.4\times 10^{-1}$}& \cellcolor{gray!25}{$2.2\times 10^{-1}$} \\ \noalign{\vskip 1pt} \hline \noalign{\vskip 1pt} 
		\multirow{3}{*}{20} & VI & $1.1\times 10^{0} $& $9.1\times 10^{-1} $& $7.6\times 10^{-1} $& $6.1\times 10^{-1} $& $5.1\times 10^{-1} $ \\ 
		& DRO & \cellcolor{gray!25}{$1.0\times 10^{0}$}& \cellcolor{gray!25}{$9.0\times 10^{-1}$}& \cellcolor{gray!25}{$7.2\times 10^{-1}$}& \cellcolor{gray!25}{$5.9\times 10^{-1}$}& \cellcolor{gray!25}{$4.8\times 10^{-1}$} \\ \noalign{\vskip 1pt} \hline \noalign{\vskip 1pt} 
		\multirow{3}{*}{30} & VI & \cellcolor{gray!25}{$1.4\times 10^{0}$}& $1.2\times 10^{0} $& $1.1\times 10^{0} $& $9.7\times 10^{-1} $& $8.5\times 10^{-1} $ \\ 
		& DRO & $1.4\times 10^{0} $& \cellcolor{gray!25}{$1.2\times 10^{0}$}& \cellcolor{gray!25}{$1.1\times 10^{0}$}& \cellcolor{gray!25}{$9.5\times 10^{-1}$}& \cellcolor{gray!25}{$8.4\times 10^{-1}$} \\ \noalign{\vskip 1pt} \hline \noalign{\vskip 1pt} 
		\multirow{3}{*}{40} & VI & \cellcolor{gray!25}{$1.6\times 10^{0}$}& \cellcolor{gray!25}{$1.5\times 10^{0}$}& $1.4\times 10^{0} $& \cellcolor{gray!25}{$1.3\times 10^{0}$}& \cellcolor{gray!25}{$1.2\times 10^{0}$} \\ 
		& DRO & $1.6\times 10^{0} $& $1.5\times 10^{0} $& \cellcolor{gray!25}{$1.4\times 10^{0}$}& $1.3\times 10^{0} $& $1.2\times 10^{0} $ \\ \noalign{\vskip 1pt} \hline \noalign{\vskip 1pt} 
		\multirow{3}{*}{50} & VI & $1.8\times 10^{0} $& \cellcolor{gray!25}{$1.7\times 10^{0}$}& \cellcolor{gray!25}{$1.6\times 10^{0}$}& $1.6\times 10^{0} $& \cellcolor{gray!25}{$1.4\times 10^{0}$} \\ 
		& DRO & \cellcolor{gray!25}{$1.8\times 10^{0}$}& $1.7\times 10^{0} $& $1.6\times 10^{0} $& \cellcolor{gray!25}{$1.6\times 10^{0}$}& $1.4\times 10^{0} $ \\ \noalign{\vskip 1pt} \hline \noalign{\vskip 1pt} 
	\end{tabular} 
\end{table}

\begin{figure*}[t!]
	\centering
	\hspace{-6pt} \subfigure[Impact of the Wasserstein radius on the suboptimality and predictibility risk]{\label{fig:bdd:index} \includegraphics[width=0.32\columnwidth]{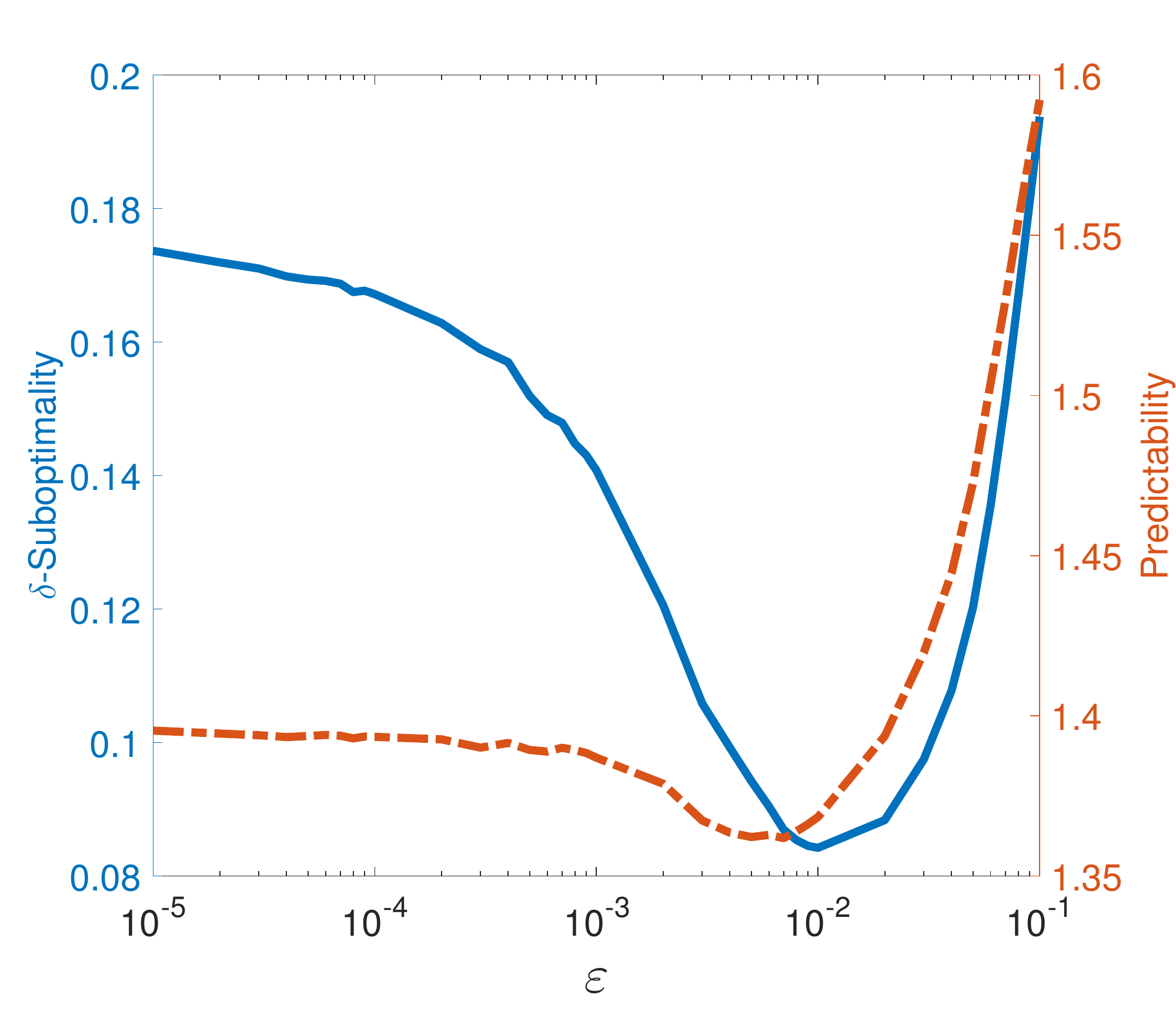}}
	\subfigure[Suboptimality learning curve]{\label{fig:bdd:learn:sub} \includegraphics[width=0.32\columnwidth]{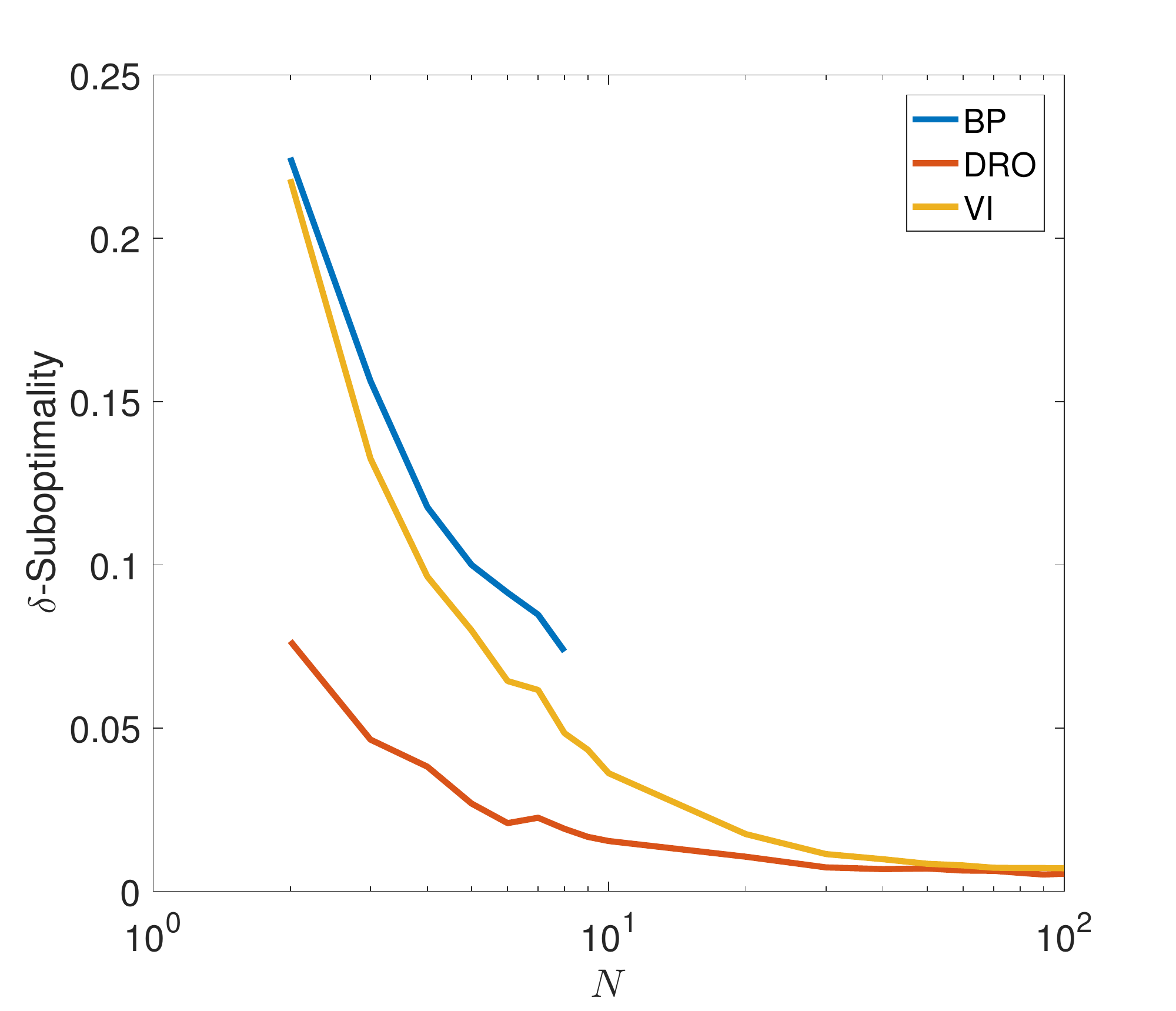}}
	\subfigure[Predictibility learning curve]{\label{fig:bdd:learn:pred} \includegraphics[width=0.32\columnwidth]{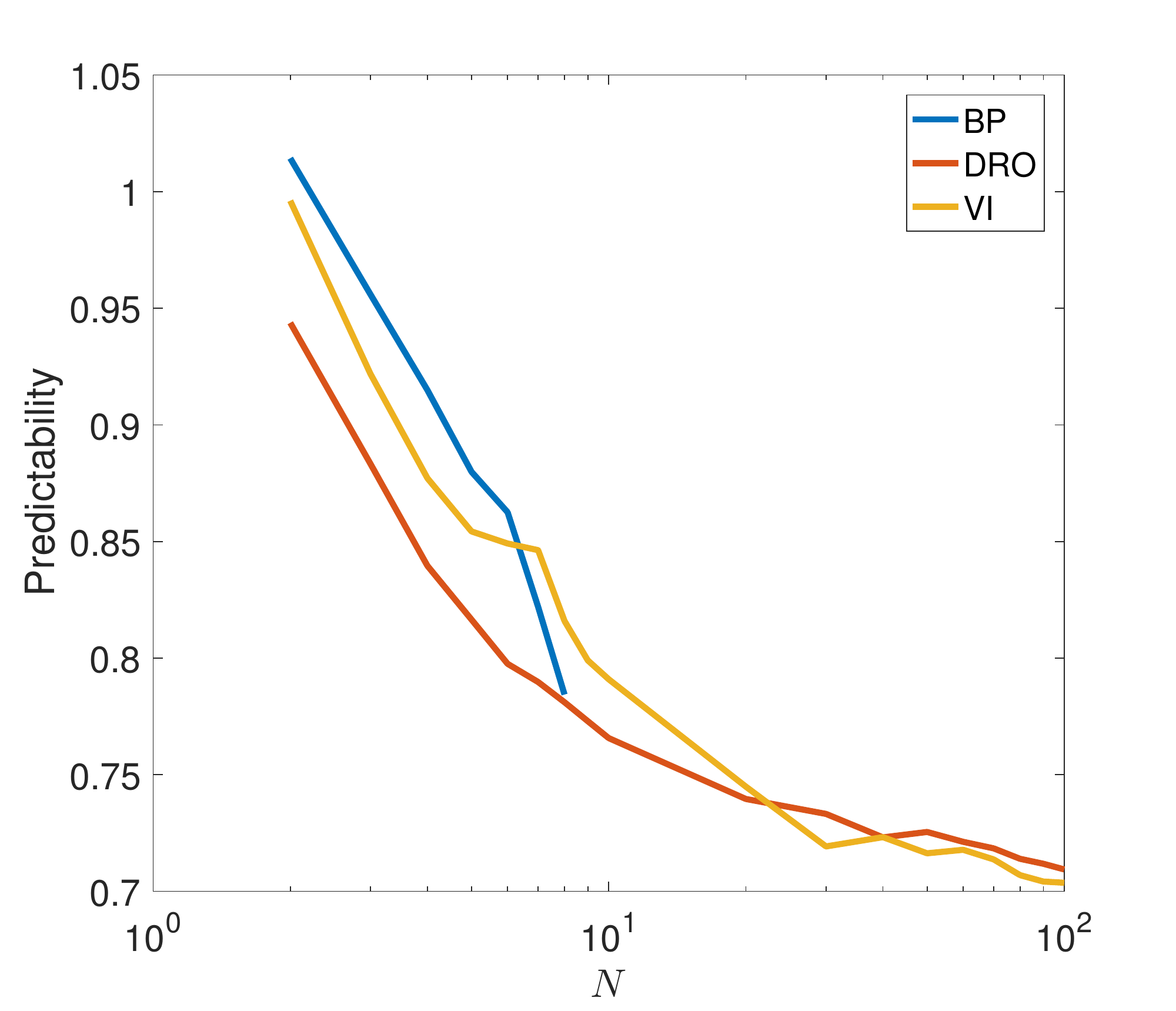}} 
	\caption{Out-of-sample suboptimality and predictability risks in the presence of imperfect consistent training and test samples}
	\label{fig:bdd}
\end{figure*}

\subsubsection{Bounded Rationality}
\label{subsubsec:bdd-ration}
Consider the exact same setting as in Section~\ref{subsubsec:meas} but assume now that the agent suffers from bounded rationality. Specifically, assume that the agent selects random $\delta$-suboptimal decisions that are perfectly measured by the observer. This means that the training samples are generated as in Section~\ref{subsubsec:meas}, but the imperfection of the training responses is now explained by the agent's bounded rationality instead of the observer's noisy measurements. At this point one may wonder why a new interpretation of the imperfections should impact the observer's inverse optimization problem. In the following we will assume, however, that the observer is aware of the agent's bounded rationality, knows the value of $\delta$ and aims to predict the agent's actual suboptimal decisions. Therefore, the DRO approach to inverse optimization is subject to two changes:
\begin{itemize}
	\item The observer minimizes the bounded rationality loss~\eqref{bounded-ratioinality-loss} instead of the suboptimality loss \eqref{loss}. 

	\item The test samples are generated in the same way as the training samples. Thus, the test responses no longer constitute perfect minimizers of~\eqref{opt_basic} but represent random $\delta$-suboptimal solutions. 
	
\end{itemize}
Recall that the bounded rationality loss favors hypotheses that correctly predict $\delta$-suboptimal responses. We also highlight that the observer's inverse optimization problem with bounded rationality loss can be reformulated as a tractable linear program by virtue of Corollary~\ref{cor:lin:bdd-r}. Moreover, by using imperfect test samples, the out-of-sample risk is now measured under the  distribution of an {\em imperfect} signal-response pair, which is the correct performance criterion given that the observer aims to predict imperfect responses.

In analogy to Section~\ref{subsubsec:meas}, the impact of the Wasserstein radius on the out-of-sample suboptimality and predictability risk is shown in Figure~\ref{fig:bdd:index}. The suboptimality and predictability learning curves of different estimators are visualized in Figures~\ref{fig:bdd:learn:sub} and \ref{fig:bdd:learn:pred}, respectively. Here, the DRO estimator is defined in terms of the bounded rationality loss, while the VI and BP estimators are computed as in Section~\ref{subsubsec:meas} and are thus not corrected for the agent's bounded rationality. The impact of the signal and response dimensions on the out-of-sample suboptimality and predictability risk are reported in Tables~\ref{table:subopt:bounded} and~\ref{table:pred:bounded}, respectively. Unless otherwise stated, all experiments are parameterized exactly as in Section~\ref{subsubsec:meas}. Here, the DRO estimator systematically attains the lowest suboptimality risk, while the VI estimator almost always wins in terms of predictability risk, even though only by a small margin.

\begin{table}
	\centering 
	\caption{Out-of-sample suboptimality risk in the presence of bounded rationality}
	\label{table:subopt:bounded}
	\begin{tabular}{cclllll}\hline 
		& & \multicolumn{5}{c}{$m$} \\ \cline{3-7} 
		$n$ & Methods & 10 & 20 & 30 & 40 & 50  \\ \noalign{\vskip 1pt} \hline \noalign{\vskip 1pt} 
		\multirow{3}{*}{10} & VI & $4.3\times 10^{-2} $& $2.5\times 10^{-2} $& $1.1\times 10^{-2} $& $5.7\times 10^{-3} $& $4.5\times 10^{-3} $ \\ 
		& DRO & \cellcolor{gray!25}{$2.5\times 10^{-2}$}& \cellcolor{gray!25}{$1.1\times 10^{-2}$}& \cellcolor{gray!25}{$6.4\times 10^{-3}$}& \cellcolor{gray!25}{$2.9\times 10^{-3}$}& \cellcolor{gray!25}{$1.8\times 10^{-3}$} \\ \noalign{\vskip 1pt} \hline \noalign{\vskip 1pt} 
		\multirow{3}{*}{20} & VI & $7.7\times 10^{-2} $& $7.8\times 10^{-2} $& $6.3\times 10^{-2} $& $3.7\times 10^{-2} $& $1.8\times 10^{-2} $ \\ 
		& DRO & \cellcolor{gray!25}{$5.3\times 10^{-2}$}& \cellcolor{gray!25}{$4.4\times 10^{-2}$}& \cellcolor{gray!25}{$3.4\times 10^{-2}$}& \cellcolor{gray!25}{$2.1\times 10^{-2}$}& \cellcolor{gray!25}{$9.2\times 10^{-3}$} \\ \noalign{\vskip 1pt} \hline \noalign{\vskip 1pt} 
		\multirow{3}{*}{30} & VI & $1.2\times 10^{-1} $& $1.3\times 10^{-1} $& $1.2\times 10^{-1} $& $8.9\times 10^{-2} $& $6.4\times 10^{-2} $ \\ 
		& DRO & \cellcolor{gray!25}{$7.9\times 10^{-2}$}& \cellcolor{gray!25}{$7.8\times 10^{-2}$}& \cellcolor{gray!25}{$6.9\times 10^{-2}$}& \cellcolor{gray!25}{$5.2\times 10^{-2}$}& \cellcolor{gray!25}{$3.9\times 10^{-2}$} \\ \noalign{\vskip 1pt} \hline \noalign{\vskip 1pt} 
		\multirow{3}{*}{40} & VI & $1.4\times 10^{-1} $& $1.9\times 10^{-1} $& $1.7\times 10^{-1} $& $1.4\times 10^{-1} $& $1.2\times 10^{-1} $ \\ 
		& DRO & \cellcolor{gray!25}{$9.7\times 10^{-2}$}& \cellcolor{gray!25}{$1.2\times 10^{-1}$}& \cellcolor{gray!25}{$1.1\times 10^{-1}$}& \cellcolor{gray!25}{$9.5\times 10^{-2}$}& \cellcolor{gray!25}{$8.2\times 10^{-2}$} \\ \noalign{\vskip 1pt} \hline \noalign{\vskip 1pt} 
		\multirow{3}{*}{50} & VI & $1.9\times 10^{-1} $& $2.0\times 10^{-1} $& $2.1\times 10^{-1} $& $2.0\times 10^{-1} $& $1.7\times 10^{-1} $ \\ 
		& DRO & \cellcolor{gray!25}{$1.2\times 10^{-1}$}& \cellcolor{gray!25}{$1.3\times 10^{-1}$}& \cellcolor{gray!25}{$1.4\times 10^{-1}$}& \cellcolor{gray!25}{$1.4\times 10^{-1}$}& \cellcolor{gray!25}{$1.2\times 10^{-1}$} \\ \noalign{\vskip 1pt} \hline \noalign{\vskip 1pt} 
	\end{tabular} 
\end{table}

\begin{table}
	\centering 
	\caption{Out-of-sample predictability risk in the presence of bounded rationality}
	\label{table:pred:bounded}
	\begin{tabular}{cclllll}\hline 
		& & \multicolumn{5}{c}{$m$} \\ \cline{3-7} 
		$n$ & Methods & 10 & 20 & 30 & 40 & 50  \\ \noalign{\vskip 1pt} \hline \noalign{\vskip 1pt} 
		\multirow{3}{*}{10} & VI & $8.2\times 10^{-1} $& \cellcolor{gray!25}{$5.8\times 10^{-1}$}& \cellcolor{gray!25}{$4.5\times 10^{-1}$}& \cellcolor{gray!25}{$3.5\times 10^{-1}$}& \cellcolor{gray!25}{$3.2\times 10^{-1}$} \\ 
		& DRO & \cellcolor{gray!25}{$7.9\times 10^{-1}$}& $5.9\times 10^{-1} $& $4.6\times 10^{-1} $& $3.7\times 10^{-1} $& $3.4\times 10^{-1} $ \\ \noalign{\vskip 1pt} \hline \noalign{\vskip 1pt} 
		\multirow{3}{*}{20} & VI & \cellcolor{gray!25}{$1.2\times 10^{0}$}& \cellcolor{gray!25}{$1.1\times 10^{0}$}& \cellcolor{gray!25}{$8.9\times 10^{-1}$}& \cellcolor{gray!25}{$7.2\times 10^{-1}$}& \cellcolor{gray!25}{$5.8\times 10^{-1}$} \\ 
		& DRO & $1.2\times 10^{0} $& $1.1\times 10^{0} $& $9.2\times 10^{-1} $& $7.4\times 10^{-1} $& $6.0\times 10^{-1} $ \\ \noalign{\vskip 1pt} \hline \noalign{\vskip 1pt} 
		\multirow{3}{*}{30} & VI & \cellcolor{gray!25}{$1.5\times 10^{0}$}& \cellcolor{gray!25}{$1.4\times 10^{0}$}& \cellcolor{gray!25}{$1.2\times 10^{0}$}& \cellcolor{gray!25}{$1.1\times 10^{0}$}& \cellcolor{gray!25}{$9.4\times 10^{-1}$} \\
		& DRO & $1.5\times 10^{0} $& $1.4\times 10^{0} $& $1.3\times 10^{0} $& $1.1\times 10^{0} $& $9.6\times 10^{-1} $ \\ \noalign{\vskip 1pt} \hline \noalign{\vskip 1pt} 
		\multirow{3}{*}{40} & VI & \cellcolor{gray!25}{$1.7\times 10^{0}$}& \cellcolor{gray!25}{$1.6\times 10^{0}$}& \cellcolor{gray!25}{$1.5\times 10^{0}$}& \cellcolor{gray!25}{$1.4\times 10^{0}$}& \cellcolor{gray!25}{$1.3\times 10^{0}$} \\ 
		& DRO & $1.8\times 10^{0} $& $1.7\times 10^{0} $& $1.6\times 10^{0} $& $1.4\times 10^{0} $& $1.3\times 10^{0} $ \\ \noalign{\vskip 1pt} \hline \noalign{\vskip 1pt} 
		\multirow{3}{*}{50} & VI & \cellcolor{gray!25}{$1.9\times 10^{0}$}& \cellcolor{gray!25}{$1.8\times 10^{0}$}& \cellcolor{gray!25}{$1.7\times 10^{0}$}& \cellcolor{gray!25}{$1.6\times 10^{0}$}& \cellcolor{gray!25}{$1.5\times 10^{0}$} \\
		& DRO & $1.9\times 10^{0} $& $1.9\times 10^{0} $& $1.8\times 10^{0} $& $1.7\times 10^{0} $& $1.6\times 10^{0} $ \\ \noalign{\vskip 1pt} \hline \noalign{\vskip 1pt} 
	\end{tabular} 
\end{table}

\subsection{Learning a Quadratic Hypothesis}
\label{subsec:quadratic}
A fundamental problem in marketing is to understand the purchasing decisions of consumers, which is an essential prerequisite for estimating demand functions. In this section we study the inverse optimization problem of a marketing manager (the observer) aiming to learn the quadratic utility function that best explains the purchasing decisions of a consumer (the agent). This problem setup is inspired by~\cite{keshavarz2011imputing}.

\subsubsection{Consistent Noisy Measurements}
\label{subsub:quad:1}

{\bf Decision problem of the agent:}  Assume that there are $n$ products with prices $s \in \R_+^n$. The agent aims to select a basket of goods $x\in\R_+^n$  that minimizes the true objective function $F(s,x)=\inner{s}{x}- U(x)$, where $\inner{s}{x}$ represents the purchasing costs, while the concave quadratic function $U(x) \Let -\inner{x}{Q_{xx}\opt x}-\inner{q\opt}{x}$ captures the utility of consumption. The positive definite matrix $Q_{xx}\opt$ is constructed as follows. Sample a square matrix $A$ uniformly form $[-1,1]^{n\times n}$, denote by $R$ the orthogonal matrix consisting of the orthonormal eigenvectors of $(A+A\tr)/ 2$ and set $Q_{xx}\opt\Let R\tr D R$, where $D$ is a diagonal matrix whose main diagonal is sampled uniformly form $[0.2,1]^n$. Moreover, the gradient $q\opt$ is sampled uniformly from $[-2,0]^n$. Finally, define the agent's feasible set as $\X(s) = [0,5]^n$, which can be brought to the standard form of Assumption~\ref{As:conic} by setting
\[
	W=(\mathbb I, -\mathbb I)\tr \in\R^{2n\times n},~ H=(0,0)\tr\in\R^{2n\times m},~ h=5\cdot (0,-1)\tr\in\R^{2n} \text{ and } \mathcal K=\R_+^{2n}.
\]
As usual, for a fixed signal $s$, we denote the optimal value of the agent's true decision problem by $z\opt(s)$.

{\bf Generation of training samples:} Signals follow the uniform distribution on the support set $\S=[0,1]^n$, which can be brought to the standard form of Assumption~\ref{As:conic} by setting $C=(\mathbb I,-\mathbb I)\tr\in\R^{2n\times n}$, $d=(0,-1)\tr\in\R^{2n}$ and $\mathcal C=\R_+^{2n}$. As in Section~\ref{subsubsec:meas}, we assume that the agent's response to $s$ is noisy and constitutes a random $\delta$-suboptimal solution to~\eqref{opt_basic} for $\delta=0.2$. Thus, $x$ is obtained as a solution of the auxiliary problem
\[
	\min_{x\in\X(s)} \left\{ \inner{x}{Q_{\rm rand}x} : \inner{x}{Q_{xx}\opt x} + \inner{s}{x} + \inner{q\opt}{x} \leq z\opt(s) + \delta \right\},
\]
which minimizes a convex quadratic cost over the set of all $\delta$-suboptimal solutions to~\eqref{opt_basic}. Specifically, $Q_{\rm rand}$ is a diagonal matrix whose main diagonal is sampled uniformly form $[0,1]^n$. As $(s,x)\in\Xi$ by construction, the measurement noise is consistent with the know support of the exact signal-response pairs. The imperfect consistent training samples $(\sd_i, \xd_i)$, $i\le N$, are now generated independently using the above procedure.

{\bf Decision problem of the observer:} The observer aims to identify the best quadratic hypothesis of the form $F_\theta(s,x)\Let  \inner{x}{Q_{xx}x} + \inner{x}{s}+\inner{q}{x}$, where the parameter $\theta=(Q_{xx}, q)$ ranges over the search space
\[
\Theta=\Big\{ \theta=(Q_{xx}, q)\in \mathbb \R^{n\times n}\times \R^n ~:~ {Q_{xx} \succeq 0} \Big\}.
\]
Note that no hypothesis $F_\theta(s,x)$, $\theta\in\Theta$, can vanish identically due to the term $\inner{x}{s}$. Note also that the agent's true objective function corresponds to $\theta\opt=(Q_{xx}\opt,q\opt)\in\Theta$. We assume that the observer minimizes the suboptimality loss~\eqref{loss}, uses the expected value to measure risk and solves the distributionally robust inverse optimization problem~\eqref{DRO-inv-opt} over a $2$-Wasserstein ball around the empirical distribution on the training samples, where the $2$-norm is used as the transportation cost on $\Xi$. By Theorem~\ref{thm:Quad}, the emerging inverse optimization problem is conservatively approximated by the tractable semidefinite program~\eqref{cvar-quad}.

{\bf Out-of-sample risk:} The quality of an estimator $\wh \theta_N$ obtained from~\eqref{cvar-quad} is measured by its out-of-sample risk $\EE^{\P_{\rm out}} (\ell_{\wh \theta_N})$ both with respect to the predictability loss~\eqref{loss_predictability} and the suboptimality loss~\eqref{loss}, where $\P_{\rm out}$ represents the distribution of a single test sample $(s,x)$ independent of the training samples, and where $x$ is an exact (non-noisy) response to $s$ in~\eqref{opt_basic}. More precisely, the out-of-sample risk is evaluated approximately using 1{,}000 independent test samples from $\P_{\rm out}$.

{\bf Results:} All numerical results are averaged across $100$ independent problems instances $\{Q_{xx}\opt, q\opt \}$. Under the assumption that the signal and response dimensions are set to $n = 10$ and there are $N=20$ training samples, Figure~\ref{fig:quad:meas} shows how the out-of-sample risk of the optimal estimator $\wh \theta_N(\eps)$ obtained from~\eqref{DRO-inv-opt} changes with the Wasserstein radius $\eps$. As in Section~\ref{subsubsec:meas}, this experiment suggests that both the predictability and suboptimality risks can be reduced by using a distributionally robust inverse optimization model.

\begin{figure*}
	\centering
	\hspace{-6pt} \subfigure[Imperfect consistent training samples and perfect test samples]{\label{fig:quad:meas} \includegraphics[width=0.3\columnwidth]{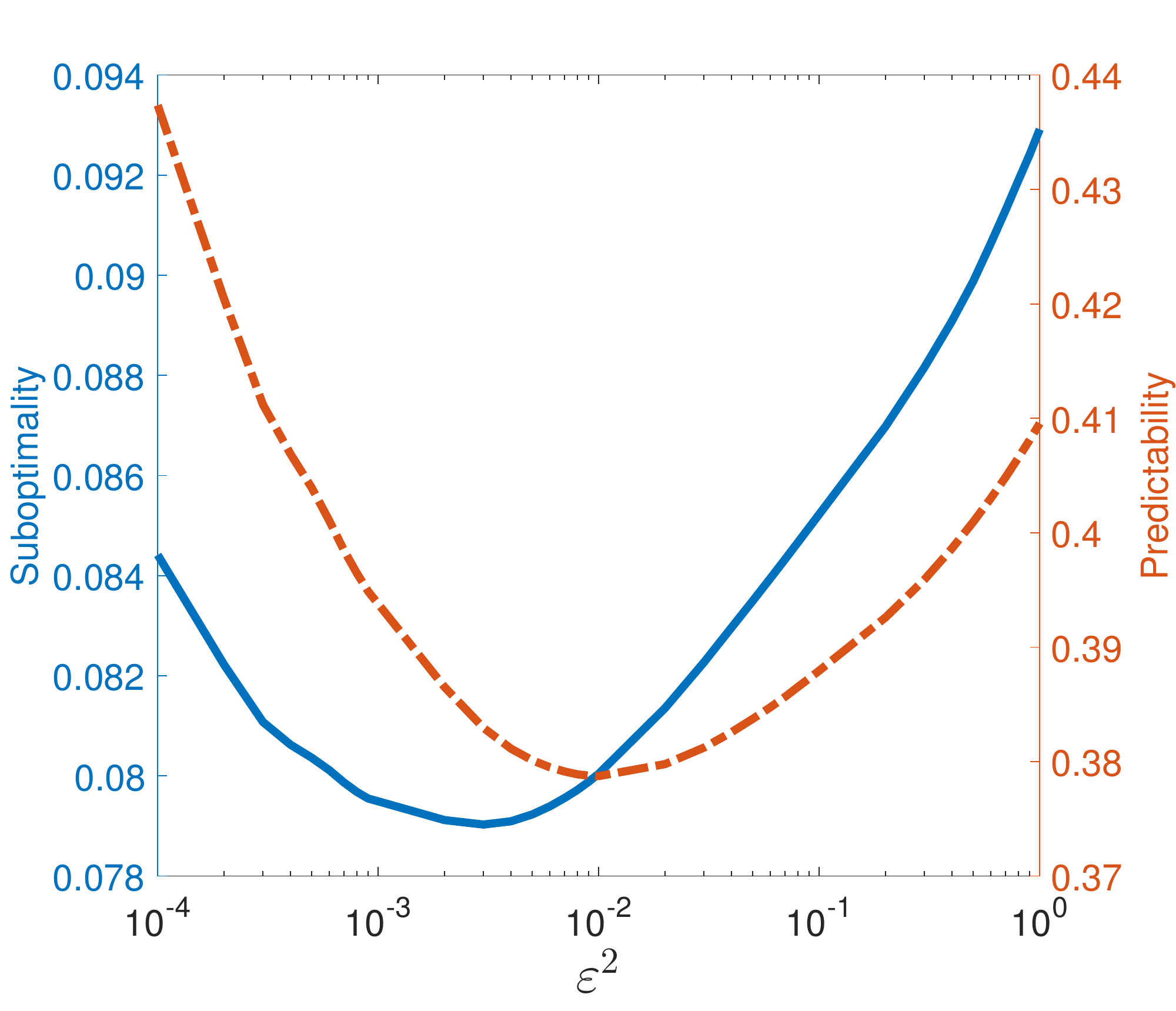}}\quad
	\subfigure[Imperfect inconsistent training samples and perfect test samples]{\label{fig:quad:meas2} \includegraphics[width=0.3\columnwidth]{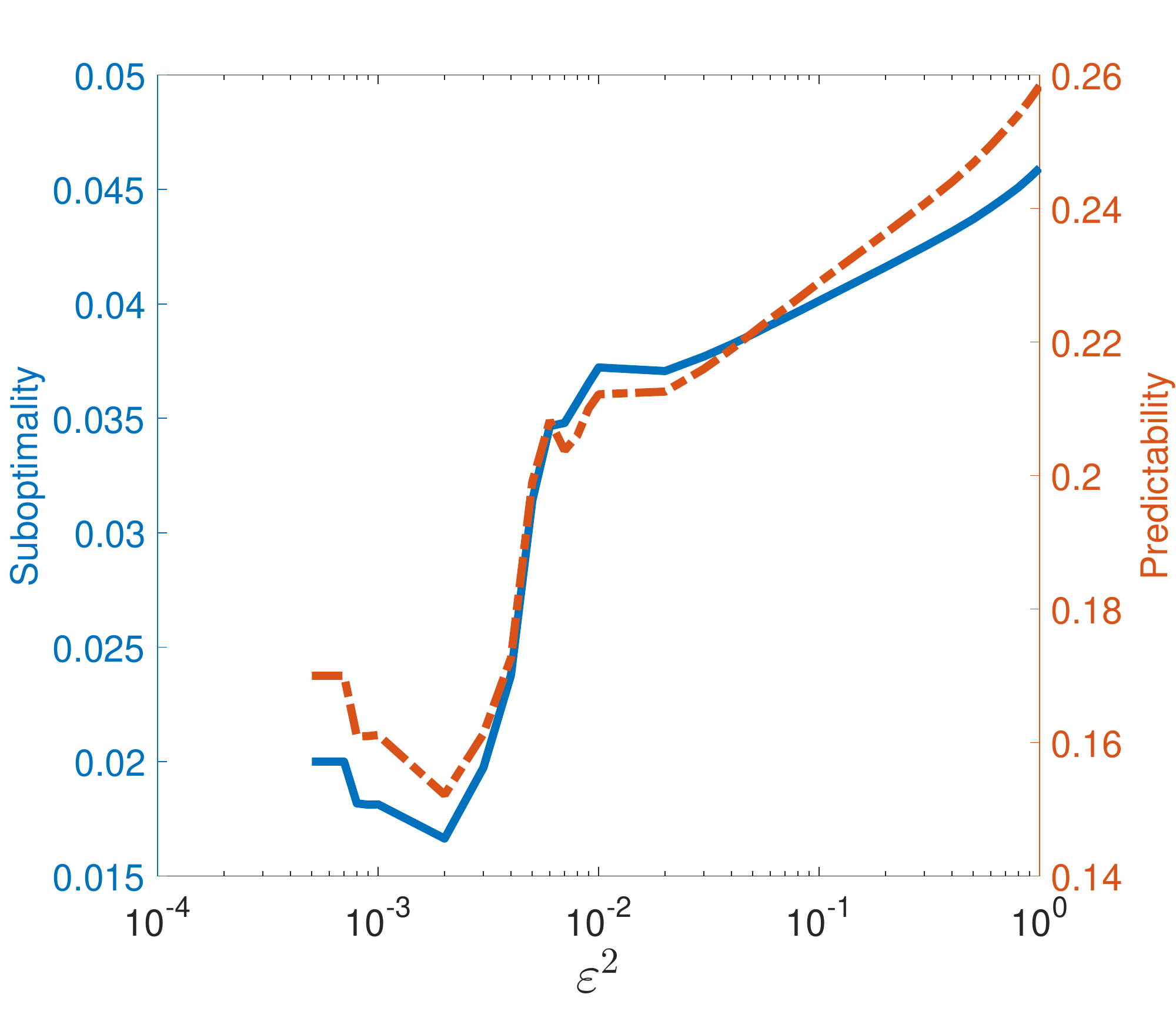}}\quad
	\subfigure[Model uncertainty]{\label{fig:quad:mismatch} \includegraphics[width=0.3\columnwidth]{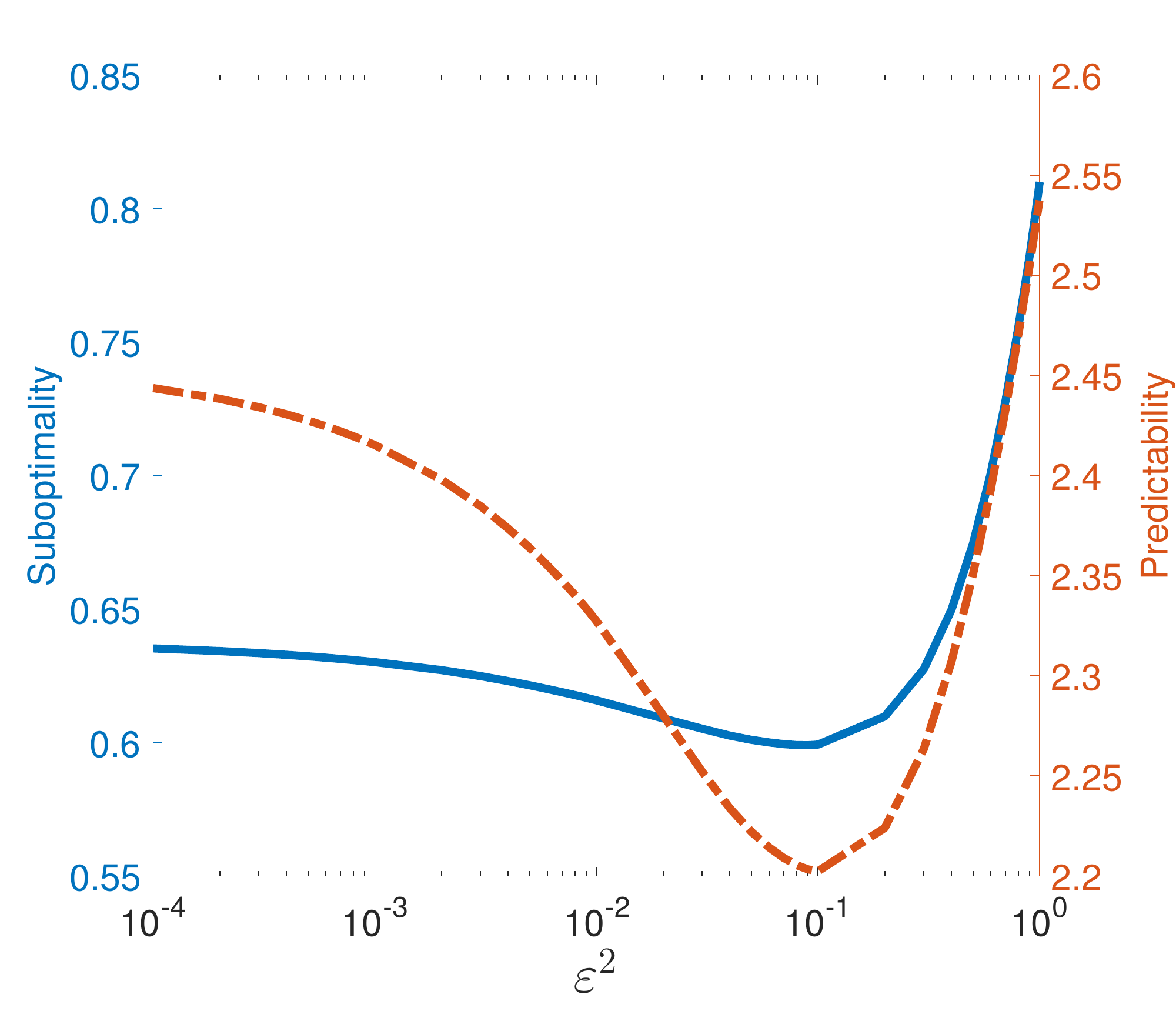}} 
	\caption{Impact of the Wasserstein radius on the out-of-sample suboptimality and predictability risk}
	\label{fig:quad}
\end{figure*}

\subsubsection{Inconsistent Noisy Measurements}
Assume now that the agent's optimal solutions to~\eqref{opt_basic} are corrupted by additive measurement noise that follows a uniform distribution on $[-0.1,0.1]^n$. Otherwise, we consider the exact same experimental setup as in Section~\ref{subsub:quad:1}. Under this premise, it is likely that some training responses are infeasible in~\eqref{opt_basic}, that is, $\xd_i\notin\X(\sd_i)$ and---{\em a fortiori}---$(\sd_i,\xd_i)\notin\Xi$ for some $i\le N$. These problematic training samples are inconsistent with the known support of the perfect signal-response pairs, and the corresponding empirical distribution $\Pem$ fails to be supported on $\Xi$. Thus, for all sufficiently small values of $\eps$ there exists no distribution $\Q$ on $\Xi$ with $\Wass{2}{\Q}{\Pem} \le \eps$, implying that the Wasserstein ball $\ball{2}{\Pem}{\eps}$ is empty, in which case the distributionally robust inverse optimization problem \eqref{cvar-quad} becomes meaningless. Figure~\ref{fig:quad:meas2} visualizes the out-of-sample suboptimality and predictability risk of the optimal estimator $\wh \theta(\eps)$ as a function of $\eps$. Note that for any fixed~$\eps$ the figure reports the out-of-sample risk averaged only across those problem instances $\{Q_{xx}\opt, q\opt \}$ for which $\ball{2}{\Pem}{\eps}\neq \emptyset$. Inspecting the results at the instance level, we observe that the out-of-sample risk is typically minimized by the smallest value of $\eps\geq 0$ for which $\ball{2}{\Pem}{\eps}\neq \emptyset$ (this is not evident from the aggregate results shown in Figure~\ref{fig:quad:meas2}). We conclude that a distributionally robust approach may be necessary for consistency.


 \subsubsection{Model Uncertainty}
Assume next that the agent's objective function is not contained in the set of hypotheses $F_\theta(s,x)$, $\theta \in \Theta$, but that the signals and the agent's responses are unaffected by noise. Specifically, in analogy to \cite{keshavarz2011imputing}, we assume that the true utility function is given by $U(x) \Let \inner{1}{\sqrt{Ax-b}}$, where $A$ is a diagonal matrix whose main diagonal is sampled uniformly from $[0.5,1]^n$, while $b$ is sampled uniformly from $[0,0.25]^n$. The square root is applied componentwise and evaluates to $-\infty$ for negative arguments. Otherwise, we consider the exact same experimental setup as in Section~\ref{subsub:quad:1}. Figure~\ref{fig:quad:mismatch} shows the out-of-sample suboptimality and predictability risk of the optimal estimator $\wh \theta(\eps)$ as a function of $\eps$, indicating that the best results are obtained for strictly positive Wasserstein radii, which enable the observer to combat over-fitting to the training samples.

\subsubsection{Comparison of Different Data-Driven Inverse Optimization Schemes}
In practice, the Wasserstein radii offering the lowest out-of-sample risk must be estimated from the training samples only. To this end, the DRO approach calibrates $\eps$ via $k$-fold cross validation as in Section~\ref{subsec:linear}. Another estimator is obtained by solving the {\em empirical risk minimization} (ERM) problem~\eqref{saa} using the empirical mean and the suboptimality loss~\eqref{loss}. This approach effectively mimics the DRO approach but sets $\eps=0$, which can be viewed as a trivial data-driven strategy to calibrate the Wasserstein radius.\footnote{We did not consider the ERM approach in Section~\ref{sec:lin} because it coincides with the VI approach for linear hypotheses.} The VI approach also disregards ambiguity and minimizes the empirical first-order loss by solving the semidefinite program 
\begin{align*}
\begin{array}{lll}
\text{\rm minimize} & \displaystyle \frac{1}{N}\sum\limits_{i = 1}^{N} |r_i| & \\
\text{\rm subject to} & \inner{W \widehat x_i - H \widehat s_i - h}{\gamma_i} \leq r_i & \forall i \leq N \\
& W^\top \gamma_i = 2 Q_{xx} x + \wh s_i  +q& \forall i \leq N \\
& \gamma_i \geq 0 & \forall i \leq N \\
& Q_{xx} \succeq 0,
\end{array}
\end{align*}
see \cite[Theorem~3]{bertsimas2012inverse}. As in Section~\ref{subsubsec:meas}, the absolute values of the residuals $r_i$ in the objective can be dropped whenever the training samples are consistent with~$\Xi$. We exclude the BP approach from this experiment because it leads to severely intractable mixed integer semidefinite programming problems.

All three approaches described above search over the {\em parametric} space of quadratic hypotheses. If the observer suspects that the true utility function fails to be quadratic, however, she may prefer a {\em non-parametric} approach that models the gradient of the utility function as a vector field~$f\in\mathcal H^n$, where~$\mathcal H$ represents a reproducing kernel Hilbert space of real-valued functions on $\R_+^n$, which is induced by a symmetric and positive definite kernel function $k:\R_+^n\times\R_+^n\rightarrow \R$. As~$\mathcal H^n$ is typically infinite-dimensional and contains multiple candidate gradients $f\in\mathcal H^n$ that explain the training data, it has been suggested in \cite[Section~5]{bertsimas2012inverse} to minimize the Hilbert norm $\sum_{i=1}^n\|f_i\|^2_\mathcal H$ of $f$ subject to the constraint that the residuals of the first-order optimality condition at the training samples satisfy $ \frac{1}{N}\sum_{i = 1}^{N} |r_i| \leq \kappa$ for some prescribed threshold~$\kappa\ge 0$. This amounts to finding the smoothest (with respect to the kernel function~$k$) candidate gradient $f\in\mathcal H^n$ that explains the training data to within accuracy~$\kappa$. By leveraging a generalized representer theorem, the resulting infinite-dimensional optimization problem can be reformulated as the tractable quadratic program
\begin{subequations}
\label{eq:non-parametric-correct}
\begin{align}
\label{eq:non-parametric}
\begin{array}{lll}
\text{\rm minimize} & \displaystyle \sum\limits_{i=1}^n \inner{e_i}{\alpha K \alpha^\top e_i} & \\
\text{\rm subject to} & \inner{W \widehat x_i - H \widehat s_i - h}{\gamma_i} \leq r_i & \forall i \leq N \\
& W^\top \gamma_i = \widehat s_i - \alpha Ke_i & \forall i \leq N \\
& \gamma_i \geq 0 & \forall i \leq N \\
& \displaystyle \frac{1}{N}\sum\limits_{i = 1}^{N} |r_i| \leq \kappa,
\end{array}
\end{align}
where $K\in \R^{N\times N}$ denotes the kernel matrix with entries $K_{ij}=k(\xd_i,\xd_j)$, $e_i$ stands for the $i$-th standard basis vector in a space of appropriate dimension, and $\alpha\in\R^{n\times N}$ denotes a decision variable that encodes the partial derivatives of the utility function via $f_i(x)=\sum_{j=1}^N \alpha_{ij}k(\xd_j,x)$; see \cite[Theorem~5]{bertsimas2012inverse}. 

In the inverse optimization context studied here, however, the above non-parametric VI approach suffers from two shortcomings that are not addressed in~\cite{bertsimas2012inverse}.

\begin{itemize}
\item[(i)] The inverse optimization problem~\eqref{eq:non-parametric} aims to learn the gradient field $f$ of the unknown utility function $U$. As the Hessian matrix of $U$ must be symmetric, the gradient field $f$ must satisfy
\[
	\frac{\partial f_i(x)}{\partial x_j} = \frac{\partial f_j(x)}{\partial x_i} \qquad \forall x\in\R_+^n,~\forall i,j=1,\ldots, n.
\]
By Stokes' theorem, this condition is necessary and sufficient to ensure that the utility function $U$ can be reconstructed uniquely (modulo an additive constant) from $f$. Specifically, this condition guarantees that the utility function is defined unambiguously through the line integral $U(x)=U(0)+\int_C \inner{f(x')}{{\rm d}x'}$, where $C$ represents an arbitrary piecewise smooth curve in $\R_+^n$ connecting $0$ and $x$. 
\item[(ii)] While the inverse optimization problem~\eqref{eq:non-parametric} represents a tractable quadratic program, the resulting gradient field $f$ may induce a non-concave utility function $U$, which means that the agent's objective function $F(s,x)=\inner{s}{x}-U(x)$ may have multiple local minima. Thus, even though learning $f$ is easy, predicting the optimal response $x$ to a given signal $s$ may require the solution of a hard non-convex optimization problem, which may severely complicate extensive out-of-sample tests. Similarly, evaluating the suboptimality loss $F(s,x)$ at a fixed signal-response pair is intractable.
\end{itemize}

Here we address the first challenge by using the polynomial kernel function $k(x,x')=(c\inner{x}{x'}+1)^p$, which allows us to include the missing symmetry conditions in~\eqref{eq:non-parametric} by appending the linear equality constraints
\begin{equation}
\label{eq:curl-free}
	\sum_{k=1}^N \Big( \alpha_{ik} [\xd_k]_j - \alpha_{jk} [\xd_k]_i\Big) \prod_{t=1}^{q-1} [\xd_k]_{l_t}=0 \quad \forall i,j=1,\ldots, n,~\forall 1\le l_1\le\cdots\le l_{q-1}\le n,~\forall q=1,\ldots,p.
\end{equation}
\end{subequations}
The second challenge does not have a simple remedy, and therefore we abandon the ideal goal to solve~\eqref{opt_basic} to global optimality. Instead, for any given signal $s$, we use the gradient field $f$ obtained from~\eqref{eq:non-parametric-correct} directly to find a local solution of~\eqref{opt_basic} via the classical subgradient descent algorithm \cite[Chapter~3]{bubeck2015convex}, where the step size is set to~$10^{-3}$ and an initial feasible solution is selected uniformly at random from $\X(s)$. Note that the focus on local minima in prediction is consistent with the use of the first-order loss~\eqref{loss_bertsimas} in inverse optimization because the first-oder loss cannot distinguish local and global optima and thus fails to penalize training samples $(\sd_i,\xd_i)$ where $\xd_i$ is a locally optimal but globally suboptimal response to $\sd_i$.

In our experiments we determine the parameter $c$ of the polynomial kernel via 5-fold cross validation. Once the gradient field $f$ has been inferred from~\eqref{eq:non-parametric-correct}, we construct the corresponding utility function by integrating $f$ along the straight line $C$ between 0 and $x$ with parameterization $g(t) = tx$ for $t \in [0,1]$, that is, we set
\begin{align*}
	U(x) &=U(0)+\int_C \inner{f(x')}{{\rm d}x'}= U(0) + \int_0^1 \inner{f(tx)}{x}\, \diff t  = U(0) + \int_0^t \sum_{i=1}^n \sum_{j=1}^N x_i \alpha_{ij}k(\xd_j,tx)\, \diff t\\
	&  = U(0) + \sum_{i=1}^n \sum_{j=1}^N  x_i \alpha_{ij} \int_0^1 (c\inner{tx}{\xd_j}+1)^p\, \diff t = U(0) + \sum_{i=1}^n \sum_{j=1}^N  x_i\alpha_{ij} \frac{\left(c\inner{x}{\widehat x_j}+1\right)^{p+1} - 1}{c (p+1) \inner{x}{\widehat x_j}}.
\end{align*}
Table~\ref{table:quad} reports the out-of-sample suboptimality and predictability risks, respectively, for the DRO, VI and ERM estimators. The non-parametric VI approach is exclusively used in the presence of model uncertainty, which is the only scenario in which it has a chance to outperform the more parsimonious parametric approaches. Our results show that the  DRO estimator consistently attains the lowest suboptimality and predictability risk among all parametric approaches. We emphasize that the suboptimality and predicatability risk of the non-parametric VI approach are evaluated with respect to the local minimum identified by the subgradient descent algorithm and are therefore largely meaningless. In fact, we observed that the suboptimality risk becomes even negative for certain instances in which the global minimum of~\eqref{opt_basic} could not be found. This artefact explains the low suboptimality risk of the non-parametric VI approach with $p=3$. 
Note also that for $p\ge 4$ the number of symmetry conditions \eqref{eq:curl-free} explodes, and thus \eqref{eq:non-parametric-correct} can no longer be solved. We also reconfirm our earlier observation that injecting robustness reduces out-of-sample risk.

\begin{table}
	\centering 
	\caption{Data-driven inverse optimization: Out-of-sample suboptimality and predictability risk of the VI, ERM and DRO approaches in different experimental settings} 
	\begin{tabular}{llll}\hline 
		&  Methods & Suboptimality & Predictability \\ \hline 
		\multirow{2}{*}{Consistent noisy} & VI (Parametric)& $3.2\times 10^{-1}$ & $1.3\times 10^{0}$ \\ 
		\multirow{2}{*}{measurements} & ERM & $8.6\times 10^{-2}$ & $4.7\times 10^{-1}$ \\ 
		& DRO & \cellcolor{gray!25}$7.9\times 10^{-2}$ & \cellcolor{gray!25}$3.9\times 10^{-1}$ \\ \hline 
		\multirow{2}{*}{Inconsistent noisy} &  VI (Parametric)& $9.2\times 10^{-2}$ & $6.3\times 10^{-1}$ \\ 
		\multirow{2}{*}{measurements} & ERM & unbounded & unbounded \\ 
		& DRO &\cellcolor{gray!25}$4.0\times 10^{-2}$ &\cellcolor{gray!25}$2.4\times 10^{-1}$ \\ \hline
		\multirow{5}{*}{Model uncertainty } & VI (Parametric)& $8.3\times 10^{-1}$ & $2.8\times 10^{0}$ \\ 
		& VI (Non-parametric, $p=2$) & $1.9 \times 10^{0}$ & $4.8\times 10^{0}$ \\ 
		& VI (Non-parametric, $p=3$) & \cellcolor{gray!25}$4.0 \times 10^{-1}$ & $6.4\times 10^{0}$ \\ 
		& ERM & $6.2\times 10^{-1}$ & $2.4\times 10^{0}$ \\ 
		& DRO & $6.0\times 10^{-1}$ & \cellcolor{gray!25}$2.2\times 10^{0}$ \\ \hline 
	\end{tabular} 
	\label{table:quad}
\end{table}
}

\begin{flushleft}
	{\bfseries Acknowledgments:} \\ This work was supported by the Swiss National Science Foundation grant BSCGI0\_157733.
\end{flushleft}

\appendix

\section{Convex Hypotheses}
\label{app:convex}
{
Consider the class $\F$ of hypotheses of the form $F_\theta(s,x)\Let \inner{\theta}{\Psi(x)}$, where each component function of the feature map $\Psi:\R^n\rightarrow\R^d$ is convex, and where the weight vector $\theta$ ranges over a convex closed search space $\Theta\subseteq \R^d_+$. Thus, by construction, $F_\theta(s,x)$ is convex in $x$ for every fixed  $\theta\in\Theta$. In the remainder we will assume without much loss of generality that the transformation from the signal-response pair $(s,x)$ to the signal-feature pair $(s,\Psi(x))$ is Lipschitz continuous with Lipschitz modulus~1, that is, we require
\begin{equation}
\label{eq:lipschitz}
	\| (s_1,\Psi(x_1)) - (s_2,\Psi(x_2)) \| \leq \| (s_1,x_1) - (s_2,x_2) \|\quad \forall (s_1,x_1), (s_2,x_2)\in\R^m\times\R^n.
\end{equation}
Before devising a safe conic approximation for the inverse optimization problem~\eqref{DRO-inv-opt} with convex hypotheses, we recall that the conjugate of a function $f:\R^n\rightarrow\R$ is defined through $f^*(z)=\sup_{x\in\R^n} \inner{z}{x}-f(x)$.

\begin{Thm}[Convex hypotheses and suboptimality loss]
	\label{thm:conv}
	Assume that $\F$ represents the class of convex hypotheses induced by the feature map $\Psi$ and with a convex closed search space $\Theta\subseteq \R_+^d$ and that Assumption~\ref{As:conic} holds. If the observer uses the suboptimality loss~\eqref{loss} and measures risk using the $\cvar$ at level $\alpha\in(0,1]$, then the following convex program provides a safe approximation for the distributionally robust inverse optimization problem~\eqref{DRO-inv-opt} over the $1$-Wasserstein ball:
	\begin{align}
	\label{cvar-conv}
	\begin{array}{llll} \text{\em minimize} & \displaystyle \tau+\frac{1}{\alpha}\left({\varepsilon} \lambda  + {1 \over N}\sum\limits_{i = 1}^{N} r_i\right) \vspace{1mm}\\
	\text{\em subject to} & \theta\in\Theta
	,\;\;	\lambda \in\mathbb R_+,\;\;\tau, r_i\in\mathbb R,\;\; \phi_{i1},\phi_{i2}\in \mathcal C^*, \;\;\gamma_{i} \in \mathcal K^* ,\;\; z_{ij}\in\R^n &\forall i \le N,\forall j\le d\\ 
	& \displaystyle \sum_{j=1}^d \theta_j \Psi^*_j(z_{ij}/\theta_j) + \inner{\theta}{\Psi(\widehat x_i)} + \inner{\phi_{i1}}{C \widehat s_i - d}- \inner{\gamma_i}{H \widehat s_i + h} \leq r_i + \tau&\forall i \le N \vspace{1mm}\\
	& \displaystyle \sum_{j=1}^d z_{ij} = W^\top \gamma_i &\forall i \le N \vspace{1mm}\\		
	& \inner{C \widehat s_i-d}{\phi_{i2}} \leq r_i &\forall i \le N \vspace{1mm}\\
	& \left\|\begin{pmatrix} H^\top \gamma_i - C^\top \phi_{i1} \\ \theta \end{pmatrix} \right\|_* \le \lambda,\;\; \left\|\begin{pmatrix} C^\top \phi_{i2} \\ 0 \end{pmatrix} \right\|_* \le \lambda &\forall i \le N.
	\end{array}
	\end{align}
\end{Thm}

Note that Theorem~\ref{thm:conv} remains valid if $(\sd_i, \xd_i)\notin\Xi$ for some $i\le N$.

\begin{proof}[Proof of Theorem~\ref{thm:conv}]
	As in the proof of Theorem~\ref{thm:lin} one can show that the objective function of the inverse optimization problem~\eqref{DRO-inv-opt} is equivalent to~\eqref{wc_cvar}. In the remainder, we derive a safe conic approximation for the (intractable) subordinate worst-case expectation problem
	\begin{align}
	\label{wc-exp-conv}
	\sup_{\Q \in \ball{1}{\Pem}{\eps}}  \EE^\Q\big[\max\{\ell_\theta(s,x) - \tau, 0\} \big].
	\end{align} 
	To this end, note that the suboptimality loss $\ell_\theta(s,x) = \inner{\theta}{\Psi(x)} - \min_{y \in \mathbb X(s)} \inner{\theta}{\Psi(y)}$ depends on $x$ only through $\Psi(x)$. This motivates us to define a lifted suboptimality loss $\ell^\Psi_\theta(s,\psi) = \inner{\theta}{\psi} - \min_{y \in \mathbb X(s)} \inner{\theta}{\Psi(y)}$, where $\psi$ represents an element of the feature space~$\R^d$, and the empirical distribution $\widehat \P_N^\Psi = \frac{1}{N}\sum_{i=1}^N \delta_{(\widehat s_i, \Psi(\widehat x_i))}$ of the signal-feature pairs. Moreover, we denote by $\mathbb B^1_{\varepsilon} (\widehat \P_N^\Psi)$ the 1-Wasserstein ball of all distributions on $\mathbb S \times \mathbb R^d$ that have a distance of at most $\eps$ from $\widehat \P_N^\Psi$. In the following we will show that the worst-case expectation
	\begin{align} \label{lifted}
	\sup_{\mathbb Q \in \mathbb B^1_{\varepsilon} (\widehat \P_N^\Psi)} \EE^\Q\big[\max\{\ell_\theta(s,\psi) - \tau, 0\} \big]
	\end{align} 
	on the signal-feature space $\mathbb S \times \mathbb R^d$ provides a tractable upper bound on the worst-case expectation~\eqref{wc-exp-conv}. 
	By Definition~\ref{def:wass}, each distribution $\Q\in \mathbb B^1_{\varepsilon} (\widehat \P_N)$ corresponds to a transportation plan $\Pi$, that is, a joint distribution of two signal-response pairs $(s,x)$ and $(s',x')$ under which $(s,x)$ has marginal distribution $\Q$ and $(s',x')$ has marginal distribution $\widehat \P_N$. By the law of total probability, the transportation plan can be expressed as $\Pi=\frac{1}{N}\sum_{i=1}^N \delta_{(\sd_i,\xd_i)}\otimes \Q_i$, where $\Q_i$ denotes the conditional distribution of $(s,x)$ given $(s',x')=(\sd_i,\xd_i)$, $i\le N$, see also~\cite[Theorem~4.2]{ref:MohKun-14}. Thus, we the worst-case expectation~\eqref{wc-exp-conv} satisfies
	\begin{align*}
	& \begin{array}{rll}
	\displaystyle \sup_{\Q \in \ball{1}{\Pem}{\eps}}  \EE^\Q\big[\max\{\ell_\theta(s,x) - \tau, 0\} \big] =& \sup\limits_{\Q^i} & \displaystyle\frac{1}{N} \sum\limits_{i=1}^N \int_{\Xi} \max\{\ell_\theta(s,\Psi(x)) - \tau, 0\} \, \Q^i(\diff s, \diff x) \\
	& \st & \displaystyle \frac{1}{N} \sum\limits_{i=1}^N \int_{\Xi} \| (s,x) - (\widehat s_i, \widehat x_i ) \| \, \Q^i(\diff s, \diff x) \leq \eps \\
	& & \displaystyle\int_{\Xi} \Q^i(\diff s, \diff x) = 1 \qquad \forall i \leq N \\
	\leq & \sup\limits_{\Q^i} & \displaystyle\frac{1}{N} \sum\limits_{i=1}^N \int_{\Xi} \max\{\ell_\theta(s,\Psi(x)) - \tau, 0\} \, \Q^i(\diff s, \diff x) \\
	& \st & \displaystyle \frac{1}{N} \sum\limits_{i=1}^N \int_{\Xi} \| (s,\Psi(x)) - (\widehat s_i, \Psi(\widehat x_i) ) \| \, \Q^i(\diff s, \diff x) \leq \eps \\
	& & \displaystyle\int_{\Xi} \Q^i(\diff s, \diff x) = 1 \qquad \forall i \leq N \\
	\leq & \sup\limits_{\Q^i} & \displaystyle\frac{1}{N} \sum\limits_{i=1}^N \int_{\mathbb S \times \R^d} \max\{\ell_\theta(s,\psi) - \tau, 0\} \, \Q^i(\diff s, \diff \psi) \\
	& \st & \displaystyle \frac{1}{N} \sum\limits_{i=1}^N \int_{\mathbb S \times \R^d} \| (s,\psi) - (\widehat s_i, \Psi(\widehat x_i) ) \| \, \Q^i(\diff s, \diff \psi) \leq \eps \\
	& & \displaystyle\int_{\mathbb S \times \R^d} \Q^i(\diff s, \diff \psi) = 1 \qquad \forall i \leq N,
	\end{array}
	\end{align*}
	where the first inequality follows from~\eqref{eq:lipschitz}, while the second inequality follows from relaxing the implicit condition that the signal-feature pair $(s,\psi)$ must be supported on $\{(s,\Psi(x)):(s,x)\in\Xi\}\subseteq\S\times\R^d$. Using a similar reasoning as before, the last expression is readily recognized as the worst-case expectation~\eqref{lifted}. Thus,~\eqref{lifted} provides indeed an upper bound on~\eqref{wc-exp-conv}. Duality arguments borrowed from~\cite[Theorem~4.2]{ref:MohKun-14} imply that the infinite-dimensional linear program~\eqref{lifted} admits a strong dual of the form
	\begin{align}
	\label{semiinf-conv2}
	\begin{array}{cll}
	\inf\limits_{\lambda \geq 0, r_i} & \displaystyle \lambda \varepsilon + \frac{1}{N} \sum_{i=1}^N r_i & \\
	\text{s.t.} & \sup\limits_{(s,y) \in \Xi, \psi \in \R^d} \inner{\theta}{\psi - \Psi(y)} - \tau - \lambda \| (s,\psi) - (\widehat s_i, \Psi(\widehat x_i)) \| \leq r_i & \forall i \leq N \\
	&\sup\limits_{s \in \mathbb S, \psi \in \R^d} -\lambda \| (s,\psi) - (\widehat s_i, \Psi(\widehat x_i)) \| \leq r_i & \forall i \leq N.
	\end{array} 
	\end{align}
	By using the definitions of $\mathbb S$ and $\X(s)$ put forth in Assumption~\ref{As:conic}, the $i$-th member of the first constraint group in~\eqref{semiinf-conv2} is satisfied if and only if the optimal value of the maximization problem
	\begin{align}
	\label{conic-conv}
	\begin{array}{cl}
	\sup\limits_{s,y,\psi} & \inner{\theta}{\psi - \Psi(y)} - \tau - \lambda \| (s,\psi) - (\widehat s_i, \Psi(\widehat x_i)) \| \vspace{1mm}\\
	\st &   C s  \gec{\coneXi} d,\; Wy \gec{\coneX} Hs + h
	\end{array}
	\end{align}
	does not exceed $r_i$. A tedious but routine calculation shows that the dual of \eqref{conic-conv} can be represented as
	\begin{align}
	\label{conic-conv-dual}
	\begin{array}{cll}
	\inf\limits_{p_i, q_i, \gamma_i, \phi_{i1}, z_{ij}} & \displaystyle \sum_{j=1}^d \theta_j \Psi^*_j(z_{ij}/\theta_j) - \tau + \inner{\theta}{\Psi(\widehat x_i)} + \inner{\phi_{i1}}{C \widehat s_i - d} - \inner{\gamma_i}{H \widehat s_i + h} \\
	\text{s.t.} & \sum_{j=1}^d z_{ij} = W^\top \gamma_i \\
	& \| (H^\top \gamma_i - C^\top \phi_{i1}, \theta) \|_* \leq \lambda \\
	& \gamma_i \in \mathcal K^*, \phi_{i1} \in \mathcal C^*.
	\end{array}
	\end{align}
	Note that the perspective functions $\theta_j \Psi^*_j(z_{ij}/\theta_j)$ in the objective of~\eqref{conic-conv-dual} emerge from taking the conjugate of $\theta_j \Psi_j(y)$. Thus, for $\theta_j=0$ we must interpret $\theta_j \Psi^*_j(z_{ij}/\theta_j)$ as an indicator function in $z_{ij}$ which vanishes for $z_{ij}=0$ and equals $\infty$ otherwise. By weak duality, \eqref{conic-conv-dual} provides an upper bound on \eqref{conic-conv}. We conclude that the $i$-th member of the first constraint group in~\eqref{semiinf-conv2} is satisfied whenever the dual problem~\eqref{conic-conv-dual} has a feasible solution whose objective value does not exceed $r_i$. A similar reasoning shows that the $i$-th member of the second constraint group in~\eqref{semiinf-conv2} holds if and only if there exists $\phi_{i2} \in \coneXi^*$ such that
	\begin{align*}
	&\begin{array}{clll} &  \inner{C\sd_i-d}{\phi_{i2}} \leq r_i \quad \text{and} \quad \left\|\begin{pmatrix}C\tr\phi_{i2} \\ 0\end{pmatrix} \right\|_* \le \lambda.
	\end{array}
	\end{align*}
	Thus, the worst-case expectation~\eqref{wc-exp-conv} is bounded above by the optimal value of the finite convex program
	\begin{align*}
	\begin{array}{clll} \Inf{} & \displaystyle {\eps} \lambda  + {1 \over N}\sum\limits_{i = 1}^{N} r_i \vspace{1mm}\\
	\st & 	\lambda \in\mathbb R_+,\;r_i\in\R,\; \phi_{i1},\phi_{i2}\in \coneXi^*, \;\gamma_{i}\in \coneX^* &\forall i \le N \\
	& \sum_{j=1}^d \theta_j \Psi^*_j(z_{ij}/\theta_j) + \inner{\theta}{\Psi(\widehat x_i)} + \inner{\phi_{i1}}{C \widehat s_i - d} - \inner{\gamma_i}{H \widehat s_i + h} \leq r_i + \tau&\forall i \le N \vspace{1mm}\\
	& \sum_{j=1}^d z_{ij} = W^\top \gamma_i &\forall i \le N \vspace{1mm}\\		
	& \inner{C \widehat s_i-d}{\phi_{i2}} \leq r_i\vspace{1mm}\\
	& \left\|\begin{pmatrix} H^\top \gamma_i - C^\top \phi_{i1} \\ \theta \end{pmatrix} \right\|_* \le \lambda,\;\; \left\|\begin{pmatrix} C^\top \phi_{i2} \\ 0 \end{pmatrix} \right\|_* \le \lambda &\forall i \le N.
	\end{array}
	\end{align*}
	The claim then follows by substituting this convex program into~\eqref{wc_cvar}. 
\end{proof}
}
	
\bibliographystyle{siam}	
\bibliography{ref} 

\end{document}